%% file: main.tex
\documentclass[preprint,12pt]{elsarticle}

\usepackage[a4paper,margin=3cm]{geometry}
\usepackage{amsmath,amssymb,amsthm,mathtools}
\usepackage{hyperref}
\usepackage{enumitem}
\usepackage{graphicx}
\usepackage{tikz} 
\usepackage[nameinlink,noabbrev]{cleveref}
\usepackage{ dsfont }
\usepackage{subcaption}
\usepackage{tikz}
\usetikzlibrary{calc}
\usetikzlibrary{arrows.meta}
\usepackage{float}

\theoremstyle{plain}
\newtheorem{theorem}{Theorem}[section]
\newtheorem{proposition}[theorem]{Proposition}
\newtheorem{lemma}[theorem]{Lemma}
\newtheorem{corollary}[theorem]{Corollary}

\theoremstyle{definition}
\newtheorem{definition}[theorem]{Definition}

\theoremstyle{remark}

\begin{document}

\begin{frontmatter}

\title{Generating Functions and the Entropy Hierarchy of Strongly Connected Digraphs}

\author{Rostislav Klech}
\ead{Rostislav.Klech@math.slu.cz}
\affiliation{organization={Mathematical Institute in Opava, Silesian University in Opava},
            addressline={Na Rybníčku 626},
            city={Opava},
            postcode={746 01},
            country={Czech Republic}}

\begin{abstract}
We present a generating-function approach to the topological entropy of finite strongly connected digraphs, working throughout with digraphs in which loops are allowed but multiple edges are excluded. Using dominant singularities of path-generating functions, we recover by new analytic--combinatorial arguments the known first and second positive-entropy minimizers from the previously studied loopless spectral setting, and extend the corresponding extremal statements to the present framework.

For the class $\mathcal{SC}_{m+1}(m)$ of strongly connected digraphs with $m$ vertices and $m+1$ edges, we introduce a unified $(t,k_1,k_2)$-butterfly parametrization. The entropy of $\mathcal{B}^{\,t}_{k_1,k_2}$ depends only on $(k_1,k_2)$ and is determined by the unique root $R\in(0,1)$ of $1-z^{k_1}-z^{k_2}=0$ via $h=-\ln R$.

This parametrization yields a detailed entropy hierarchy within $\mathcal{SC}_{m+1}(m)$. We introduce the Pyramidal Entropy Diagram, determine the entropy order completely for $m\leq 7$, and identify its first structural bifurcation at $m=8$. We further establish maximal stable initial and terminal segments consisting of ten entropy minima and two entropy maxima, respectively, and derive an explicit formula for the minimum order required to realize a positive entropy not exceeding a prescribed threshold.
\end{abstract}


\end{frontmatter}

\section{Introduction}

For a finite directed graph (digraph) $G$, one may associate the
(one-sided or two-sided) edge shift of finite type, whose admissible
words are precisely the directed paths in $G$; see Lind and
Marcus~\cite{LindMarcus1995} for background. In this setting, the topological
entropy is determined by the asymptotic growth of the number $p(n)$
of directed paths of length $n$, namely
\[
h(G)=\limsup_{n\to\infty}\frac{1}{n}\ln p(n).
\]

A classical route to computing $h(G)$ proceeds via the adjacency
matrix and Perron--Frobenius theory. In particular, if $A$ is the
adjacency matrix of a strongly connected digraph $G$, then
\[
h(G)=\ln \rho(A),
\]
where $\rho(A)$ denotes the spectral radius of $A$. In this paper,
we develop an alternative analytic--combinatorial framework based on
generating functions and path decompositions. We work throughout
with finite digraphs in which loops are allowed, while multiple edges
are excluded, and carry out the analysis from the foundational order
$m=2$.

Our starting point is the ordinary generating function
\[
P(z)=\sum_{n\geq 0}p(n)z^n,
\]
together with vertex-based generating functions $P_i(z)$ obtained
from the natural recurrence ``a path continues by choosing a
successor.'' This yields a system of functional equations that can
be solved in terms of rational functions. Using Pringsheim's theorem
for power series with nonnegative coefficients, we connect the
entropy to the analytic singularity structure of $P(z)$. More
precisely, for a strongly connected digraph $G$,
\[
h(G)=-\ln R,
\]
where $R$ is the radius of convergence of $P(z)$. Consequently,
$R=\rho(A)^{-1}$. Thus, determining the entropy reduces to identifying
the dominant positive singularity of the relevant rational generating
function, which we encode by a suitable topological polynomial and
the corresponding topological equation. For background on generating
functions and their analytic use in asymptotic enumeration, we refer
to Wilf~\cite{Wilf1994Gfology} and Flajolet and Sedgewick~\cite{FlajoletSedgewick2009}. For general
digraph terminology and standard facts, see Bang-Jensen and
Gutin~\cite{BangJensenGutin2009}.

The first extremal problem considered in this paper is the
minimization of positive topological entropy among strongly connected
digraphs of a fixed order $m$. Since the directed cycle is the unique
strongly connected digraph with zero entropy, this problem corresponds,
in the adjacency-matrix setting, to determining the second smallest
spectral radius among strongly connected digraphs of order $m$. Lin
and Shu~\cite{LinShu2012} studied the corresponding extremal problem within the
class of strongly connected bicyclic digraphs. Li and Zhou~\cite{LiZhou2015}
subsequently identified the digraphs attaining the second, third, and
fourth smallest spectral radii among all strongly connected digraphs
of order $m\geq 4$, in the setting without loops and multiple edges.
For further results on adjacency spectra of strongly connected
digraphs, see also~\cite{HONG201493, LIN20123663}.

In the terminology introduced below, the digraphs attaining the
smallest and second smallest positive topological entropy are
\[
\mathcal{B}^{\,m-1}_{m-1,m}
\qquad\text{and}\qquad
\mathcal{B}^{\,m-2}_{m-1,m-1},
\]
respectively. For $m\geq 4$, these correspond to the previously known
second and third spectral-radius minimizers. We recover both extremal
results, together with their uniqueness, by a generating-function
argument that does not use adjacency matrices or Perron--Frobenius
theory. Thus, in the previously studied loopless setting, the novelty
of these results lies in the proof rather than in the identification
of the extremal digraphs. Our formulation additionally allows loops
and covers the small orders beginning with $m=2$.

To describe the near-minimal edge regime, we introduce the
$(t,k_1,k_2)$-butterfly notation. The underlying structures are closely
related to the classical $\theta$- and $\infty$-digraphs used in the
spectral literature. More precisely,
\[
\mathcal{B}^{\,1}_{k_1,k_2}\cong \infty(k_1,k_2),
\]
whereas, for $t\geq 2$,
\[
\mathcal{B}^{\,t}_{k_1,k_2}
\cong
\theta(k_1-t,k_2-t,t-2).
\]
The butterfly parametrization is particularly convenient for the
generating-function approach: the parameters $k_1$ and $k_2$ directly
record the lengths of the two distinguished first-return paths, while
$t$ records the size of their common vertex set.

We show that every digraph in the class $\mathcal{SC}_{m+1}(m)$ of
strongly connected digraphs with $m$ vertices and $m+1$ edges admits
such a butterfly representation. For a butterfly digraph
$\mathcal{B}^{\,t}_{k_1,k_2}$, the entropy is determined by the unique root
$R\in(0,1)$ of
\[
1-z^{k_1}-z^{k_2}=0
\]
through the relation
\[
h\left(\mathcal{B}^{\,t}_{k_1,k_2}\right)=-\ln R.
\]
Consequently, the entropy depends only on the pair $(k_1,k_2)$ and is
independent of the intersection parameter $t$. The generating-function
viewpoint therefore provides additional structural information: the
topological polynomial can be read directly from the butterfly
parameters, and entropy comparisons reduce to comparisons within the
elementary polynomial family
\[
1-z^{k_1}-z^{k_2}.
\]

This reduction allows us to study the internal entropy hierarchy of
$\mathcal{SC}_{m+1}(m)$ beyond the first few spectral-radius minimizers
previously identified in the literature. We introduce the Pyramidal
Entropy Diagram, whose vertices represent butterfly types
$(k_1,k_2)$ and whose directed edges encode strict entropy growth.
The induced reachability relation defines a partial order compatible
with entropy. We show that the entropy ordering is total for
$m\leq 7$, whereas at $m=8$ the relative order of certain butterfly
types changes. We further determine stable initial and terminal
segments consisting of ten entropy minima and two entropy maxima.
These segments are maximal in the sense that their relative ordering
is independent of $m$, while the ordering of further butterfly types
may vary with the number of vertices.

Finally, we consider the inverse problem of determining how many
vertices are required to realize a prescribed small positive entropy.
Using the extremal butterfly family, we derive the explicit formula
\[
m_{\min}(\varepsilon)
=
\max\left\{
2,
\left\lceil
\frac{\ln\left(e^\varepsilon+1\right)}{\varepsilon}
\right\rceil
\right\}
\]
for the smallest order of a strongly connected digraph whose positive
topological entropy does not exceed $\varepsilon$.

\medskip
\noindent\textbf{Organization of the paper.}
Section~2 fixes notation and basic graph-theoretic terminology. In Section~3 we present the
generating-function method and establish the entropy--radius-of-convergence principle for
strongly connected digraphs, together with tools based on first-return decompositions.
Section~4 proves the global minimum results and develops the butterfly description and its
entropy consequences. Section~5 constructs the Pyramidal Entropy Diagram and derives the
resulting ordering and extremal statements within $\mathcal{SC}_{m+1}(m)$.

\section{Preliminaries and notation}
We work with finite directed graphs $G=(V,E)$ with $|V|=m$.
Loops are allowed, while multiple edges are not.
A \emph{directed path} or \emph{path} of length $n$ from $u$ to $v$ is a series of vertices
$P_{uv}=v_0\dots v_n$ with $v_0=u$ and $v_n=v$ such that $(v_{i-1},v_i)\in E$
for all $i=1,\dots,n$.
The path is \emph{closed} if $u=v$ and \emph{simple} if its vertices are pairwise distinct
(except that in the closed case we allow $v_0=v_n$). A closed path $P_{v_0v_n}=v_0v_1\dots v_n$ is called a \emph{first-return path to $v$}
if $v_0=v_n=v$ and $v_i\neq v$ for all $i=1,\dots,n-1$. The \emph{length of the path} $P_{uv}$ is defined as the number of its vertices minus one, i.e.:
\[ 
    |P_{uv}|=|v_0\dots v_n|:=n+1-1=n.
 \]
If $P_{u_0u_n}=u_0\dots u_n$ and  $P_{v_0v_m}=v_0\dots v_m$ with $u_n=v_0$, then we define their concatenation as
\[ 
    P_{u_0u_n}\circ P_{v_0v_n}:=u_0\dots u_nv_1\dots v_m, 
 \]
and the length as $|P_{u_0u_n}\circ P_{v_0v_m}|:= |P_{u_0u_n}|+|P_{v_0v_m}|$. Vertex $u$ is \emph{reachable} from vertex $v$ if there is a directed path from $v$ to $u$. The digraph $G$ is \emph{strongly connected} if for any $u,v\in V$ there exists a directed path from $u$ to $v$. We define the \emph{out-degree} of $v\in V$ by
\[
\deg^+(v)\;=\;\bigl|\{\,u\in V : (v,u)\in E\}\bigr|,
\]
and the \emph{in-degree} of $v$ by
\[
\deg^-(v)\;=\;\bigl|\{\,u\in V : (u,v)\in E\}\bigr|.
\]
\emph{The distance between vertices} $u$ and $v$ (precisely in this order) , $\operatorname{dist}(u,v)$, is the length of the shortest path $P_{uv}$. 
If $G = (V(G), E(G))$ and $H = (V(H), E(H))$ are two digraphs, we say that $G$ and $H$ are \emph{isomorphic}, written $G \cong H$, if there exists a bijection
\[
  \phi \colon V(G) \to V(H)
\]
such that for all $u, v \in V(G)$,
\[
  (u, v) \in E(G)
  \quad\Longleftrightarrow\quad
  \bigl(\phi(u), \phi(v)\bigr) \in E(H).
\]
If $V(H)\subseteq V(G)$ and $E(H)\subseteq E(G)$, then $H$ is a \emph{subgraph} of $G$. We also denote by $\mathcal{C}_m$ the digraph $G=(V,E)$, for which
\begin{align*}
    V &= \{v_1, v_2, \dots, v_m\}, \\
    E &= \big\{(v_i, v_{i+1}) : i \in \{1, 2, \dots, m-1\} \big\} \cup \big\{(v_m, v_1)\big\}.
\end{align*}

Let $p(n)$ denote the total number of paths of length $n$ in $G$, i.e.
\[
p(n):=\bigl|\{(v_0,\dots,v_n)\in V^{n+1} : (v_{i-1},v_i)\in E \text{ for } i=1,\dots,n\}\bigr|.
\]
We define the corresponding generating function by
\[
P(z):=\sum_{n\ge 0} p(n)\,z^n.
\]
The topological entropy of $G$ is defined by
\[
h(G):=\limsup_{n\to\infty}\frac{1}{n}\ln\big(p(n)\big).
\]
Two digraphs $G$ and $H$ are called \textit{isentropic} if $h(G) = h(H)$.

\section{The generating-function method}
In this section, we establish the analytic framework for computing topological entropy by identifying it as the negative logarithm of the radius of convergence of path-generating functions. By applying Pringsheim’s theorem, we reduce the study of dynamical complexity to the identification of the dominant singularity of rational functions derived from vertex-based recurrence relations. This methodology utilizes first-return decompositions to provide a purely combinatorial alternative to traditional spectral arguments.

\begin{theorem}[Pringsheim’s theorem]\label{pring}
Let
\[
f(z)=\sum_{n=0}^{\infty} a_n z^n
\]
be a power series with $a_n\ge 0$ for all $n$ and radius of convergence $R\in(0,\infty)$.
Then $z=R$ is a singular point of $f$.
\end{theorem}

\begin{proposition}\label{prop:entropy-radius}
Let $G$ be a strongly connected digraph. Then
\[
h(G)=-\ln R,
\]
where $R$ is the radius of convergence of the generating function $P(z)$.
\end{proposition}

Now we describe how to construct the system of recurrences associated with a digraph 
$G$.
Let $G$ be a strongly connected digraph with $m$ vertices. We define $P(z)$ as the generating function:
\[
P(z) = \sum_{n=0}^{\infty} p(n)\,z^n.
\] 

Every path of length $n$ starts at one of the vertices of the graph. Let $p_i(n)$ denote the number of paths of length $n$ that start at vertex $v_i$. We define the generating function for the number of paths in the digraph that start at vertex $v_i$ as:
\begin{equation}\label{rekurence}
    P_i(z) = \sum_{n=0}^{\infty} p_i(n)\,z^n, \quad p_i(0)=1.
\end{equation}

Let $S(v_i)$ be the \emph{set of indices of successors of vertex} $v_i$. In other words, 
\[ 
    S(v_i)=\Big\{j\in \{1,2,...,m\}:(v_i,v_j)\in E\Big\}.
 \]
The number of paths of length $n$ starting at vertex $v_i$ can then be decomposed as the sum of the number of paths of length $n-1$ starting at its successors:
\begin{equation}\label{rozklad_gen}
    p_i(n) = \sum_{j \in S(v_i)} p_j(n-1), \quad n\geq1.
\end{equation}

Multiplying (\ref{rozklad_gen}) by $z^n$ and summing over all $n \geq 1$, we obtain:

\begin{align*}
    \sum_{n=1}^{\infty} p_i(n) z^n &=  \sum_{n=1}^{\infty}\sum_{j \in S(v_i)} p_j(n-1)z^n \\
    &=\sum_{j \in S(v_i)}z\sum_{n=0}^{\infty} p_j(n)z^{n} \\
    &=\sum_{j \in S(v_i)} z P_j(z).
\end{align*}
Thus, by substituting back into \eqref{rekurence}, for each $i$ we obtain
\[
P_i(z) = 1 + \sum_{j \in S(v_i)} z\,P_j(z).
\]

 The overall generating function $P(z)$ can then be expressed as the sum of all the partial generating functions $P_i(z)$:
 \[ \label{sum:gen_func}
     P(z) = \sum_{i=1}^{m} P_i(z).
  \]

While the following Lemma~\ref{gen_func} and its Corollary~\ref{cor:one_vertex_radius} are equivalent to a well-known result \cite[Theorem V.7]{FlajoletSedgewick2009}, we provide an alternative, purely combinatorial proof. This ensures our framework remains self-contained and strictly avoids matrix methods, relying exclusively on path structures in digraphs.

\begin{lemma}\label{gen_func}
    Let $G$ be a strongly connected digraph with $m$ vertices. For each vertex $v_i$, let $P_i(z)$ be the generating function for the number of paths in $G$ that starts at $v_i$. Then all $P_i(z)$ have the same radius of convergence. 
\end{lemma}

\begin{proof}
Let $G$ be a strongly connected digraph with $m$ vertices. Since the graph is strongly connected, for every pair of vertices $v_i$ and $v_j$ there exists a shortest directed path connecting them. Denote such a path by
\[
v_i = v_{i_0} v_{i_1} \dots v_{i_k} = v_j.
\]

Recall that the generating functions satisfy the recurrence
\[
P_i(z) = 1 + \sum_{l \in S(v_i)} z\,P_l(z).
\]
In general, for every $r \in \{0,1,\dots, k-1\}$ we deduce that
\[
P_{i_r}(z)=1+\sum_{l\in S(v_{i_r})}z\,P_l(z) \geq 1 + z\,P_{i_{r+1}}(z),
\]
for $z>0$. By successive substitution, it follows that
\[
P_{i}(z) = P_{i_0}(z) \geq \sum_{l=0}^{k-1}z^l + z^k\,P_{i_k}(z) = \sum_{l=0}^{k-1}z^l + z^k\,P_j(z).
\]
Denote the radii of convergence of $P_i(z)$ and $P_j(z)$ by $R_i$ and $R_j$, respectively. 

Now, assume for the sake of contradiction that $R_i>R_j$ and choose $r_0\in(R_j,R_i)$.
Then the power series $P_j(r_0)$ diverges, since $r_0$ lies outside the disc of convergence of $P_j$.
Evaluating the inequality
\[
P_i(z)\ge \sum_{l=0}^{k-1} z^l + z^k P_j(z)
\]
at $z=r_0>0$, we obtain that the divergence of $P_j(r_0)$ forces $P_i(r_0)$ to diverge as well.
This contradicts the fact that $P_i$ converges for all $|z|<R_i$.

A similar argument applies if $R_j > R_i$ by considering a shortest directed path from $v_j$ to $v_i$.
Hence, we conclude that
\[
R_i = R_j, \quad \forall i,j\in \{1,2,...,m\}.
\]
\end{proof}

\begin{corollary}\label{cor:one_vertex_radius}
Let $G$ be a strongly connected digraph with $m$ vertices. Then the generating functions
$P_1(z),\dots,P_m(z)$ all have the same radius of convergence, say $R$.
Moreover, the total generating function
\[
P(z)=\sum_{i=1}^m P_i(z)
\]
has radius of convergence $R$ as well. In particular, to determine $R$ (and hence $h(G)$)
it suffices to compute $P_i(z)$ for any fixed vertex $v_i$.
\end{corollary}

\begin{proof}
Since $P(z)=\sum_{i=1}^m P_i(z)$ is a finite sum of power series with nonnegative coefficients,
its radius of convergence equals $\min_{1\le i\le m} R_i$, where $R_i$ is the radius of convergence of $P_i(z)$.
By Lemma~\ref{gen_func} we have $R_1=\cdots=R_m=:R$, and therefore $P(z)$ has radius of convergence $R$.
\end{proof}

\begin{lemma}\label{compare}
Let $G$ be a strongly connected digraph and let $v_i\in V$. For $n\ge 0$, let $p_i(n)$ denote the number of paths of length $n$ that start at the vertex $v_i$, and define
\[
P_i(z)=\sum_{n=0}^{\infty} p_i(n)\,z^n.
\]
Let $FR_i$ be a set of pairwise distinct first-return paths at the vertex $v_i$, and for each $fr_j\in FR_i$ let $l_j$ denote the length of $fr_j$. Define
\[
G(z):=\sum_{fr_j\in FR_i} z^{\,l_j}
\]
and
\[
Q(z):=\frac{1}{1-G(z)}.
\]
Let $R_{i}$ and $R_Q$ be the radii of convergence of $P_i(z)$ and $Q(z)$, respectively. Then
\[
R_{i}\le R_Q.
\]
\end{lemma}

\begin{proof}
Define a sequence $q(n)$ as the number of sequences $(fr_{j_1},fr_{j_2},\dots,fr_{j_t})$ such that
\[
l_{j_1}+l_{j_2}+\dots+l_{j_t}=n.
\]
Thus, $q(n)$ counts the number of closed paths of length $n$ with initial vertex $v_i$ that are composed only of elements of $FR_i$. This number is finite: for each fixed $n\in\mathds{N}$ there are only finitely many closed paths of length $n$ in a finite digraph, and hence $q(n)$ is finite.

Moreover, $q(n)$ satisfies the recurrence
\[
q(n)=\sum_{fr_j\in FR_i} q(n-l_j), \qquad q(0)=1,\qquad q(n<0)=0,
\]
analogously to the standard decomposition by the last block. In addition, for a fixed $n\in\mathds{N}$ only those elements $fr_j\in FR_i$ with $l_j\le n$ contribute to the sum, and there are only finitely many such elements. Hence all sums below are, in fact, finite.

Let
\[
Q(z)=\sum_{n=0}^{\infty} q(n)\,z^n.
\]
We show that this $Q(z)$ agrees with the function defined above. Indeed,
\begin{align*}
\sum_{n=1}^{\infty} q(n)\,z^n
&=\sum_{n=1}^{\infty}\sum_{fr_j\in FR_i} q(n-l_j)\,z^n \\
&=\sum_{fr_j\in FR_i}\sum_{n=1-l_j}^{\infty} q(n)\,z^{n+l_j} \\
&=\sum_{fr_j\in FR_i}\sum_{n=0}^{\infty} q(n)\,z^{n+l_j} \\
&=\sum_{fr_j\in FR_i} z^{l_j}\sum_{n=0}^{\infty} q(n)\,z^n \\
&=\sum_{fr_j\in FR_i} z^{l_j} Q(z).
\end{align*}
Therefore,
\[
Q(z)-q(0)=\sum_{fr_j\in FR_i} z^{l_j} Q(z)
\quad\Longrightarrow\quad
Q(z)=\frac{1}{1-\sum_{fr_j\in FR_i} z^{l_j}}.
\]

Finally, we certainly have $p_i(n)\ge q(n)$ for all $n$, and hence
\[
R_{i}\le R_Q.
\]
\end{proof}

\begin{lemma}\label{lem:R_bounds}
Let $G$ be a strongly connected digraph on $m$ vertices and let
\[
P(z)=\sum_{n=0}^{\infty} p(n)\,z^n
\]
be the generating function for the total number of paths of length $n$ in $G$.
Let $R$ be the radius of convergence of $P(z)$. Then
\[
R\in(0,1].
\]
Moreover,
\[
R=1 \Longleftrightarrow G\cong \mathcal{C}_m.
\]
\end{lemma}

\begin{proof}
Since $G$ is strongly connected, every vertex has out-degree at least $1$. Hence, for each $n\ge 0$ there exists at least one directed path of length $n$ starting at every vertex, and therefore $p(n)\ge m$.
On the other hand, from any starting vertex there are at most $m$ choices at each step, so $p(n)\le m\cdot m^n$.
Taking $n$th roots and $\limsup$ yields
\[
1\le \limsup_{n\to\infty} p(n)^{1/n}\le m,
\]
and thus
\[
R\in\left[\frac1m,\,1\right]\subset(0,1].
\]

Assume that $G\cong \mathcal{C}_m$. Let $R_i$ denote the radius of convergence of $P_i(z)$. Then for any vertex $v_i$ and any $n\in\mathds{N}$ we have $p_i(n)=1$, since $\deg^+(v_i)=1$ for all $i$. Hence
\[
P_i(z)=\sum_{n=0}^{\infty} z^n=\frac{1}{1-z} \quad\Longrightarrow\quad R_i=1,
\]
and, by Corollary~\ref{cor:one_vertex_radius}, we have
\[ 
    R_i=R=1.
 \]

Now assume that $G$ is a strongly connected digraph on $m$ vertices such that $G\ncong \mathcal{C}_m$. Then there exists a vertex $v_i$ with $\deg^+(v_i)\ge 2$, which implies the existence of at least two first-return paths $fr_1$ and $fr_2$ of lengths $l_1$ and $l_2$, respectively. Define
\[
Q(z)=\frac{1}{1-z^{l_1}-z^{l_2}}.
\]
Let $R_Q$ be the radius of convergence of $Q(z)$. Since $1-z^{l_1}-z^{l_2}$ is strictly decreasing on $(0,1)$, satisfies $(1-z^{l_1}-z^{l_2})\big|_{z=0}=1$, and equals $-1$ at $z=1$, the equation
\[
1-z^{l_1}-z^{l_2}=0
\]
has a unique solution in $(0,1)$, and hence
\[
R_Q\in(0,1).
\]
Using Lemma~\ref{compare}, we obtain
\[
R_i\in(0,1).
\]
\end{proof}

\begin{definition}\label{def:topological_polynomial}
Let $G$ be a strongly connected digraph. Let
\[
P(z)=\sum_{n=0}^{\infty} p(n)\,z^n
\]
be the generating function, where $p(n)$ denotes the number of paths of length $n$, and let $R$ be the radius of convergence of $P(z)$.
Assume that $h(G)>0$ ($G\ncong \mathcal{C}_m$), so that $R\in(0,1)$.

A polynomial $T_G(z)\in\mathbb{Z}[z]$ is called a \emph{topological polynomial of $G$}
if it satisfies the following two conditions:
\begin{enumerate}
\item $T_G(R)=0$,
\item $T_G(z)\neq 0$ for every $z\in(0,R)$.
\end{enumerate}

The equation
\[
T_G(z)=0
\]
is called the \emph{topological equation} of $G$.
\end{definition}

\begin{lemma}\label{retezec}
Let $G$ be a strongly connected digraph and suppose that there exists a sequence of vertices
\[
(v_1,v_2,\dots,v_i)
\]
such that $\deg^{+}(v_j)=1$ for every $j\in\{1,2,\dots,i-1\}$ and 
\[
S(v_j)=\{j+1\}\qquad\text{for every } j\in\{1,2,\dots,i-1\}.
\]
Then
\[
P_1(z)=\frac{1-z^{\,i-1}}{1-z}+z^{\,i-1}P_i(z).
\]
\end{lemma}

\begin{proof}
The corresponding system of recurrences is
\begin{align*}
P_1(z)&=1+zP_2(z), \\
&\vdots \\
P_j(z)&=1+zP_{j+1}(z), \\
&\vdots \\
P_{i-1}(z)&=1+zP_i(z).
\end{align*}
Solving this system by iterating the relations yields
\[
P_1(z)=\sum_{j=0}^{i-2} z^j+z^{\,i-1}P_i(z)
=\frac{1-z^{\,i-1}}{1-z}+z^{\,i-1}P_i(z).
\]
\end{proof}

\section{Global minimum on $m$ vertices}
This chapter addresses the identification of the unique strongly connected digraph of order $m$ that minimizes positive topological entropy. To systematically analyze the $m+1$ edge regime, we introduce the $(t, k_1, k_2)$-butterfly digraph family and derive its corresponding topological equations. We prove that entropy within this class is independent of the intersection parameter $t$, allowing for the constructive realization of the global minimizer.

\begin{definition}
Let $G$ be a strongly connected digraph on $m$ vertices and $k_1,k_2 \in \{1,2,\dots\}$.  We say that
$G$ is a \emph{$(t,k_1,k_2)$-butterfly digraph} $\mathcal{B}^{\,t}_{k_1,k_2}$ if
\[
  |E| = m + 1,\qquad
  k_1 + k_2 - t = m,\qquad
  1 \le t \le \min\{k_1,k_2\},
\]
and there exist two directed simple cycles $\mathcal{C}_{k_1},\mathcal{C}_{k_2}$ as a subgraphs of $G$ satisfying
\[
  \left|V(\mathcal{C}_{k_1})\right| = k_1,\quad
  \left|V(\mathcal{C}_{k_2})\right| = k_2,\quad
  \bigl|V(\mathcal{C}_{k_1})\cap V(\mathcal{C}_{k_2})\bigr| = t.
\]
\end{definition}

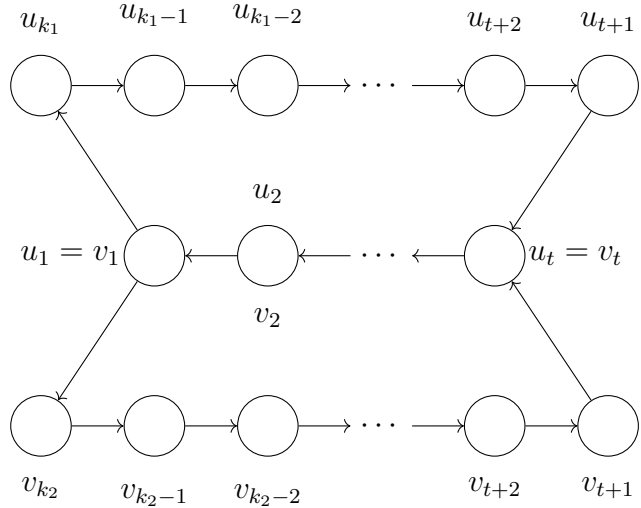
\begin{figure}[H]
\centering
\begin{tikzpicture}[scale=1.5,
    every node/.style={
      circle,
      draw,
      minimum size=8mm,
      inner sep=0.5pt,           
      label distance=0.1pt       
    }
]
  \node[label=left:{$u_1=v_1$}] (A) at (0,0) {};
  \node[label=above:{$u_2$},label=below:{$v_2$}] (B) at (1,0) {};
  \node[label=right:{$u_t=v_t$}] (D) at (3,0) {};
  \node[label=above:{$u_{k_1}$}] (E) at (-1,1.5) {};
  \node[label=above:{$u_{k_1-1}$}] (F) at (0,1.5) {};
  \node[label=above:{$u_{k_1-2}$}] (G) at (1,1.5) {};
  \node[label=above:{$u_{t+2}$}] (H) at (3,1.5) {};
  \node[label=above:{$u_{t+1}$}] (I) at (4,1.5) {};
  \node[label=below:{$v_{k_2}$}] (J) at (-1,-1.5) {};
  \node[label=below:{$v_{k_2-1}$}] (K) at (0,-1.5) {};
  \node[label=below:{$v_{k_2-2}$}] (L) at (1,-1.5) {};
  \node[label=below:{$v_{t+2}$}] (M) at (3,-1.5) {};
  \node[label=below:{$v_{t+1}$}] (N) at (4,-1.5) {};
  \node[
    shape=circle,       
    minimum size=8mm,  
    draw=none,          
    inner sep=0pt      
  ] (dots) (dots1) at ($(G)!0.5!(H)$) {$\cdots$};
  \node[
    shape=circle,      
    minimum size=8mm,   
    draw=none,          
    inner sep=0pt       
  ] (dots) (dots2) at ($(L)!0.5!(M)$) {$\cdots$};

  \node[
    shape=circle,      
    minimum size=8mm,   
    draw=none,          
    inner sep=0pt       
  ] (dots) (dots3) at ($(B)!0.5!(D)$) {$\cdots$};

  \draw[->] (A)--(E);
  \draw[->] (E)--(F);
  \draw[->] (F)--(G);
  \draw[->] (G)--(dots1);
  \draw[->] (dots1)--(H);
  \draw[->] (H)--(I);
  
  \draw[->] (I)--(D);
  \draw[->] (D)--(dots3);
  \draw[->] (dots3)--(B);
  \draw[->] (B)--(A);

  \draw[->] (A)--(J);
  \draw[->] (J)--(K);
  \draw[->] (K)--(L);
  \draw[->] (L)--(dots2);
  \draw[->] (dots2)--(M);
  \draw[->] (M)--(N);
  \draw[->] (N)--(D);
  
\end{tikzpicture}
\caption{Arbitrary $\mathcal{B}^{\,t}_{k_1,k_2}$}
\label{motyl}
\end{figure}

\begin{lemma}\label{monotonicity}
Let $n_1,n_2,m_1,m_2\in\mathds{N}$ be such that $n_1\le m_1$ and $n_2\le m_2$.
Let $p_1(z)=1-z^{n_1}-z^{n_2}$ and $p_2(z)=1-z^{m_1}-z^{m_2}$.
Then $p_1$ and $p_2$ each have exactly one root in the interval $(0,1)$; denote these roots by
$z_1$ and $z_2$, respectively. Moreover, $z_1\le z_2$, with equality only if $n_1=m_1$ and $n_2=m_2$.
\end{lemma}

\begin{lemma}\label{lem:m+1_is_butterfly}
Let $G=(V,E)$ be a strongly connected digraph with $|V|=m$ and $|E|=m+1$.
Then there exist integers $k_1,k_2\in\mathds{N}$ and $t\in\mathds{N}$ such that
\[
G\cong \mathcal{B}^{\,t}_{k_1,k_2}.
\]
\end{lemma}

\begin{proof}
Since $|E|=m+1$, we have $\sum_{u\in V}\deg^{+}(u)=m+1$.
Strong connectivity implies $\deg^{+}(u)\ge 1$ for all $u\in V$, hence there is a unique vertex
$v\in V$ with $\deg^{+}(v)=2$ and $\deg^{+}(u)=1$ for all $u\neq v$.
Let the two outgoing edges of $v$ be $(v,v_1)$ and $(v,v_2)$ (possibly $v_j=v$ if there is a loop at $v$).

For $j\in\{1,2\}$, choose a shortest directed path $P_{v_jv}$ from $v_j$ to $v$ and let its length be $k_j-1$.
Then $vv_j\circ P_{v_jv}$ is a directed simple cycle $\mathcal C_{k_j}$ of length $k_j$.

Moreover, every vertex lies on at least one of these two cycles: starting from $v$, the first step uses either the edge
$v\to v_1$ or $v\to v_2$, and after leaving $v$ the path is forced (all out-degrees are $1$) until it returns to $v$.
Thus $V=V(\mathcal C_{k_1})\cup V(\mathcal C_{k_2})$.
Let
\[
t=\bigl|V(\mathcal C_{k_1})\cap V(\mathcal C_{k_2})\bigr|.
\]
The number of common vertices $t$ is at least $1$ (since the vertex $v$ is shared) and at most $\min\{k_1,k_2\}$. Then $m=k_1+k_2-t$, and $G$ is isomorphic to a butterfly digraph $\mathcal{B}^{\,t}_{k_1,k_2}$.

\end{proof}

\begin{lemma}\label{t_omez}
    Let $m\geq2$ and $G\cong \mathcal{B}^{\,t}_{k_1,k_2}$ for some $t,k_1,$ and $k_2$. Then
    \[ 
        t\in\{1,2,\dots,m-1\}.
     \]
\end{lemma}

\begin{proof}
Without loss of generality assume $k_1\leq k_2$. We know that $t\le k_1$ and $k_1\le k_2$, hence $k_1+k_2\ge 2t$. On the other hand,
equation $k_1+k_2-t=m$ gives $k_1+k_2=m+t$, and therefore $m+t\ge 2t$, which implies
$t\le m$. If $t=m$, then necessarily $t=k_1=k_2=m$, which would yield an impossible
butterfly digraph. Consequently, for fixed $m$ we always have $t\le m-1$.
\end{proof}

\begin{lemma}\label{gen_fun_motyl}
Let $G\cong \mathcal{B}^{\,t}_{k_1,k_2}$ for some $t,k_1,$ and $k_2$. Let $w_1$ be the unique vertex of $G$ such that
$\deg^{+}(w_1)=2$, and let $P_1(z)$ denote the generating function for the numbers of directed paths of length $n$
starting at $w_1$. Then
\[
P_1(z)=\frac{N_1(z)}{(1-z)\bigl(1-z^{k_1}-z^{k_2}\bigr)},
\]
where $N_1(z)$ is a polynomial in the variable $z$.
\end{lemma}

\begin{proof}
    Label by $w_1$ the unique vertex with $\deg^+(w_1)=2$, and let $\mathcal{C}_{k_1}$ and $\mathcal{C}_{k_2}$ be the two directed simple cycles in $G$, with
\begin{align*}
    V(\mathcal{C}_{k_1})&=\{w_1,u_2,\dots,u_t,u_{t+1},\dots,u_{k_1-1},u_{k_1}\},\\
    V(\mathcal{C}_{k_2})&=\{w_1,v_2,\dots,v_t,v_{t+1},\dots,v_{k_2-1},v_{k_2}\},
\end{align*}
as in Figure~\ref{motyl}.  Hence
\[
    V(\mathcal{C}_{k_1})\cap V(\mathcal{C}_{k_2})
    =\{w_1,u_2,\dots,u_t\}
    =\{w_1,v_2,\dots,v_t\}.
\]

Let $P_i(z)$ denote the generating function for vertices that belong to both cycles $\mathcal{C}_{k_1}$ and $\mathcal{C}_{k_2}$, and let $P_{i,\mathcal{C}_{k_j}}(z)$, for $j=1,2$, denote the generating function for vertices that lie in exactly one of these cycles.  We then decompose the generating function for the entire digraph $G$ as follows:

\[
    P(z)
    = \sum_{i=1}^t P_i(z)
    + \sum_{i=t+1}^{k_1} P_{i,\mathcal{C}_{k_1}}(z)
    + \sum_{i=t+1}^{k_2} P_{i,\mathcal{C}_{k_2}}(z),
\]
and from the structure of $\mathcal{B}^{\,t}_{k_1,k_2}$ (Figure \ref{motyl}) we obtain
\[
    P_1(z)
    = 1 + z\,P_{k_1,\mathcal{C}_{k_1}}(z)
         + z\,P_{k_2,\mathcal{C}_{k_2}}(z),
\]
together with
\[
\begin{alignedat}{2}
  P_{k_1,\mathcal{C}_{k_1}}(z)&=1+z\,P_{k_1-1,\mathcal{C}_{k_1}}(z),\quad&
  P_{k_2,\mathcal{C}_{k_2}}(z)&=1+z\,P_{k_2-1,\mathcal{C}_{k_2}}(z),\\
  P_{k_1-1,\mathcal{C}_{k_1}}(z)&=1+z\,P_{k_1-2,\mathcal{C}_{k_1}}(z),&
  P_{k_2-1,\mathcal{C}_{k_2}}(z)&=1+z\,P_{k_2-2,\mathcal{C}_{k_2}}(z),\\
  &\vdots& &\vdots\\
   P_{t+1,\mathcal C_{k_1}}(z)&=1+z\,P_t(z) &
   P_{t+1,\mathcal C_{k_2}}(z)&=1+z\,P_t(z), \\
\end{alignedat}
\]
and 
\begin{align*}
    & P_t(z)=1+z\,P_{t-1}(z),\\
& \hspace{8mm}\vdots\\
& P_2(z)=1+z\,P_1(z).
\end{align*}
Hence, by Lemma \ref{retezec}
\[
    P_{k_1,\mathcal{C}_{k_1}}(z)
    = \frac{1-z^{k_1-1}}{1-z}+z^{k_1-1}P_1(z),
    \quad
    P_{k_2,\mathcal{C}_{k_2}}(z)
    = \frac{1-z^{k_2-1}}{1-z}+z^{k_2-1}P_1(z),
\]
and consequently
\begin{align*}
    P_1(z)
    &= 1 + \frac{z-z^{k_1}}{1-z}+z^{k_1}P_1(z) 
      + \frac{z-z^{k_2}}{1-z}+z^{k_2}P_1(z),\\
    P_1(z)
    &= \frac{N_1(z)}{(1-z)(1 - z^{k_1} - z^{k_2})}.
\end{align*}
\end{proof}

\begin{corollary}\label{polynom_m+1}
Let $G\cong \mathcal{B}^{\,t}_{k_1,k_2}$ for some $t$, $k_1$, and $k_2$.
Then the equation
\[
1-z^{k_1}-z^{k_2}=0
\]
has a unique solution $z_0\in(0,1)$,
\[
T_G(z)=1-z^{k_1}-z^{k_2},
\]
and
\[
h(G)=-\ln(z_0).
\]
In particular, $1-z^{k_1}-z^{k_2}=0$ is the topological equation of $G$.
\end{corollary}

\begin{proof}
Let $w_1$ denote the unique vertex of $G$ such that
$\deg^{+}(w_1)=2$. By Lemma~\ref{gen_fun_motyl}, the generating function $P_1(z)$
has the form
\[
P_1(z)
=
\frac{N_1(z)}
{(1-z)\bigl(1-z^{k_1}-z^{k_2}\bigr)}.
\]
Let $R$ denote the common radius of convergence of $P_1(z)$ and
the total generating function $P(z)$; these radii coincide by
Corollary~\ref{cor:one_vertex_radius}.

Since $G\not\cong \mathcal{C}_m$, Lemma~3.6 gives $R\in(0,1)$. By Theorem~\ref{pring}, $z=R$ is a singularity of $P_1$. Since
$P_1$ is rational, $R$ must be a zero of its displayed
denominator. The factor $1-z$ does not vanish on $(0,1)$, whereas the polynomial
\[
1-z^{k_1}-z^{k_2}
\]
has a unique zero $z_0\in(0,1)$. Hence $R=z_0$. In particular,
the zero $z_0$ cannot be cancelled by the numerator, since
otherwise it would be a removable singularity.

Moreover, since $z_0$ is the unique zero of
$1-z^{k_1}-z^{k_2}$ in $(0,1)$, this polynomial has no zero in
$(0,z_0)$. Therefore,
\[
T_G(z)=1-z^{k_1}-z^{k_2}
\]
is a topological polynomial of $G$.

Finally, by Proposition~\ref{prop:entropy-radius},
\[
h(G)=-\ln R=-\ln(z_0).
\]
Thus,
\[
1-z^{k_1}-z^{k_2}=0
\]
is the topological equation of $G$.
\end{proof}

\begin{corollary}\label{srovnavani_motyl_minimum}
Let $\mathcal{SC}_{m+1}(m)$ denote the class of all strongly connected digraphs on $m$ vertices with $m+1$ edges, considered up to isomorphism and $m\geq2$.
Then
\[
h\!\left(\mathcal{B}_{m-1,m}^{\,m-1}\right) \leq h(G)
\quad \text{for every } G\in \mathcal{SC}_{m+1}(m),
\]
with the equality if and only if $G\cong\mathcal{B}_{m-1,m}^{\,m-1}$.
In other words, the digraph $\mathcal{B}_{m-1,m}^{\,m-1}$ minimizes the topological entropy (up to isomorphism) within the class $\mathcal{SC}_{m+1}(m)$.

Moreover,
\[ 
    h\!\left(\mathcal{B}_{m-1,m}^{\,m-1}\right)<\frac{\ln2}{m-1}.
 \]
\end{corollary}

\begin{proof}
Let $G\in\mathcal{SC}_{m+1}(m)$. Without loss of generality, assume that
$k_1\le k_2$. Then, using Lemma~\ref{t_omez}, by Lemma~\ref{lem:m+1_is_butterfly} there exist integers $k_1\le m-1$, $k_2\le m$, and $t\le k_1$ such that $G\cong \mathcal{B}^{t}_{k_1,k_2}$.
A topological polynomial of $G$ is
\[
T_G(z)=1-z^{k_1}-z^{k_2},
\]
and the topological entropy is determined by the positive solution $z(k_1,k_2)$ of the topological equation
\[ 
    1-z^{k_1}-z^{k_2}=0.
 \]
Since $z(k_1,k_2)\in(0,1)$, Lemma~\ref{monotonicity} yields that the root of $1-z^{k_1}-z^{k_2}$ is maximized when $k_1$ and $k_2$ take their largest admissible values. In particular, $z(m-1,m)>z(k_1,k_2)$ for every admissible pair $(k_1,k_2)\neq(m-1,m)$, where $z(m-1,m)$ is the root of the polynomial 
\[ 
    1-z^{m-1}-z^m,
 \]
which is the topological polynomial of $\mathcal{B}_{m-1,m}^{\,m-1}$.
Moreover, since $z(m-1,m-1)=2^{-\frac{1}{m-1}}$, we have
\[
h\!\left(\mathcal{B}_{m-1,m-1}^{\,m-2}\right)=\frac{\ln 2}{m-1},
\]
and consequently
\[
h\!\left(\mathcal{B}_{m-1,m}^{\,m-1}\right)<\frac{\ln 2}{m-1}.
\]
\end{proof}

\begin{corollary}\label{cor:butterfly_t_irrelevant}
Let $k_1,k_2\ge 1$ and let $t\in\{1,2,\dots,m-1\}$.
Let $R=R(k_1,k_2)\in(0,1)$ be the unique positive root of
\[
1-z^{k_1}-z^{k_2}=0.
\]
Then
\[
h(\mathcal{B}^{\,t}_{k_1,k_2})=-\ln R,
\]
in particular the value is independent of $t$.

Moreover, if $k_1=k_2=k$, then $R=2^{-1/k}$ and hence for every $t\in\{1,\dots,k-1\}$,
\[
h(\mathcal{B}^{\,t}_{k,k})=\frac{\ln 2}{k}.
\]
\end{corollary}

\begin{proof}
By Corollary~\ref{polynom_m+1} we have
\[
h(\mathcal{B}^{\,t}_{k_1,k_2})=-\ln R,
\]
where $R\in(0,1)$ is the unique positive root of $1-z^{k_1}-z^{k_2}=0$.
In particular, the left-hand side depends only on $(k_1,k_2)$ and not on $t$,
which proves the first claim.

If $k_1=k_2=k$, then the equation becomes $1-2z^k=0$, hence
$R=2^{-1/k}$ and therefore $h(\mathcal{B}^{\,t}_{k,k})=\ln(2)/k$ for every admissible
$t$. Note that by definition of $\mathcal{B}^{\,t}_{k,k}$ we have
$t\in\{1,2,\dots,k-1\}$, i.e.\ the case $t=k$ is not admissible.
\end{proof}

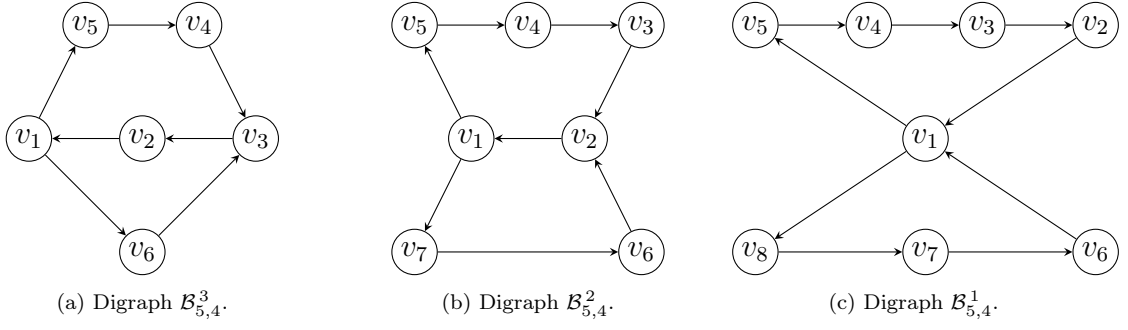
\begin{figure}[h!]
  \centering

  \begin{subfigure}[b]{0.32\textwidth}
    \centering
    \begin{tikzpicture}[>=stealth,
        node/.style={circle,draw,minimum size=6mm,inner sep=1pt}
    ]
      \node[node] (v1) at (0,0) {$v_1$};
      \node[node] (v2) at (1.5,0) {$v_2$};
      \node[node] (v3) at (3,0)   {$v_3$};
      \node[node] (v4) at (0.75,1.5) {$v_5$};
      \node[node] (v5) at (2.25,1.5) {$v_4$};
      \node[node] (v6) at (1.5, -1.5) {$v_6$};
      \draw[->] (v3)--(v2);
      \draw[->] (v2)--(v1);
      \draw[->] (v1)--(v4);
      \draw[->] (v4)--(v5);
      \draw[->] (v5)--(v3);
      \draw[->] (v1)--(v6);
      \draw[->] (v6)--(v3);
    \end{tikzpicture}
    \caption{Digraph $\mathcal{B}^{\,3}_{5,4}$.}
    \label{fig:butterfly34}
  \end{subfigure}\hfill
  %
  \begin{subfigure}[b]{0.32\textwidth}
    \centering
    \begin{tikzpicture}[>=stealth,
        node/.style={circle,draw,minimum size=6mm,inner sep=1pt}
    ]
      \node[node] (v1) at (0,0) {$v_1$};
      \node[node] (v2) at (1.5,0) {$v_2$};
      \node[node] (v5) at (-0.75,1.5)   {$v_5$};
      \node[node] (v4) at (0.75,1.5) {$v_4$};
      \node[node] (v3) at (2.25,1.5) {$v_3$};
      \node[node] (v6) at (2.25,-1.5) {$v_6$};
      \node[node] (v7) at (-0.75, -1.5) {$v_7$};
      \draw[->] (v2) -- (v1);
      \draw[->] (v3) -- (v2);
      \draw[->] (v4) -- (v3);
      \draw[->] (v5) -- (v4);
      \draw[->] (v1) -- (v5);
      \draw[->] (v1) -- (v7);
      \draw[->] (v7) -- (v6);
      \draw[->] (v6) -- (v2);
    \end{tikzpicture}
    \caption{Digraph $\mathcal{B}^{\,2}_{5,4}$.}
    \label{fig:butterfly24}
  \end{subfigure}\hfill
  %
  \begin{subfigure}[b]{0.32\textwidth}
    \centering
    \begin{tikzpicture}[>=stealth,
        node/.style={circle,draw,minimum size=6mm,inner sep=1pt}
    ]
      \node[node] (c) at (0,0) {$v_1$};

      \node[node] (t2) at (-2.25, 1.5) {$v_5$};
      \node[node] (t3) at (-0.75, 1.5) {$v_4$};
      \node[node] (t4) at ( 0.75, 1.5) {$v_3$};
      \node[node] (t5) at ( 2.25, 1.5) {$v_2$};

      \node[node] (b6) at (-2.25,-1.5) {$v_8$};
      \node[node] (b7) at ( 0.0,-1.5) {$v_7$};
      \node[node] (b8) at ( 2.25,-1.5) {$v_6$};

      \draw[->] (c)  -- (t2);
      \draw[->] (t2) -- (t3);
      \draw[->] (t3) -- (t4);
      \draw[->] (t4) -- (t5);
      \draw[->] (t5) -- (c);

      \draw[->] (c)  -- (b6);
      \draw[->] (b6) -- (b7);
      \draw[->] (b7) -- (b8);
      \draw[->] (b8) -- (c);
    \end{tikzpicture}
    \caption{Digraph $\mathcal{B}^{\,1}_{5,4}$.}
    \label{fig:newdigraph}
  \end{subfigure}

  \caption{Comparison of $\mathcal{B}^{\,3}_{5,4}$, $\mathcal{B}^{\,2}_{5,4}$, and $\mathcal{B}^{\,1}_{5,4}$.}
  \label{fig:three_digraphs}
\end{figure}

\noindent\textbf{Remark.} It is worth noting that within the class $\mathcal{SC}_{m+1}(m)$ for a fixed $m$, the topological entropy $h(G)$ acts as a complete invariant. While Corollary~\ref{cor:butterfly_t_irrelevant} establishes that the intersection parameter $t$ does not affect the entropy , the structural constraint $k_1 + k_2 - t = m$ ensures that for a fixed $m$, no two non-isomorphic digraphs can share the same ordered pair $(k_1, k_2)$ with $k_1 \le k_2$. Consequently, isentropic but non-isomorphic butterfly digraphs can only exist across different orders $m$, meaning they must have a different number of vertices and edges. This phenomenon is illustrated in Figure~\ref{fig:three_digraphs}.

\begin{lemma}\label{distance}
Let $G$ be a strongly connected digraph on $m$ vertices, and let $v\in V$ satisfy
$\deg^+(v)=k\le m$. Let $v_i$ be the vertices such that $(v,v_i)\in E$ for
$i=1,2,\dots,k$, and set $S(v)=\{1,2,\dots,k\}$. Then there exists at most one
$i\in S(v)$ such that $\operatorname{dist}(v_i,v)=m-1$. Moreover, if such an index
$i\in S(v)$ exists, then there is a simple path $P_{v_iv}$ with $|P_{v_iv}|=m-1$,
and for every $j\in S(v)\smallsetminus\{i\}$ there exists a simple path $P_{v_jv}$
with $|P_{v_jv}|\le m-2$.
\end{lemma}

\begin{proof}
Assume that there exists $i\in S(v)$ such that $\text{dist}(v_i,v)=m-1$.
By definition of dist$(\cdot,\cdot)$, there exists a shortest directed path from $v_i$ to $v$
of length $m-1$; denote it by $P_{v_iv}$.
Since $P_{v_iv}$ is shortest, it is simple, and because $|P_{v_iv}|=m-1$ in a digraph
with $m$ vertices, the path $P_{v_iv}$ visits all vertices of $G$.

In particular, for every $j\in S(v)\smallsetminus\{i\}$ the vertex $v_j$ lies on $P_{v_iv}$.
Hence the subpath of $P_{v_iv}$ from $v_j$ to $v$ is a directed path $P_{v_jv}$
of length strictly smaller than $m-1$. Therefore $|P_{v_jv}|\le m-2$, and consequently
$\text{dist}(v_j,v)\le m-2$ for all $j\neq i$.
\end{proof}

\begin{lemma}\label{lem:three_fr_dv_small}
Let $G$ be a strongly connected digraph on $m$ vertices. Suppose there exist distinct vertices $v$ and $w$
with $\deg^{+}(v)=\deg^{+}(w)=2$.  Denote by $v_1$ and $v_2$ the vertices satisfying $(v,v_1),(v,v_2)\in E$ and assume that
\[
\operatorname{d}(v):=\max_{i=1,2}\{\operatorname{dist}(v_i,v)\}\le m-2.
\]
Then $v$ admits three distinct first-return paths $fr_1,fr_2,fr_3$ such that
\[
|fr_1|\le m-1,\qquad |fr_2|\le m-1,\qquad |fr_3|\ge 1.
\]
\end{lemma}

\begin{proof}
By Lemma~\ref{distance}, there exist simple paths $P_{v_1v}$ and $P_{v_2v}$ from $v_1$ to $v$ and from $v_2$ to $v$,
respectively, with
\[
|P_{v_1v}|\le m-2,\qquad |P_{v_2v}|\le m-2.
\]
Concatenating with the edges $vv_1$ and $vv_2$ yields two first-return paths to $v$,
\[
fr_1:=vv_1\circ P_{v_1v},\qquad fr_2:=vv_2\circ P_{v_2v},
\]
and hence
\[
|fr_1|\le m-1,\qquad |fr_2|\le m-1.
\]
We proceed to construct a third first-return path $fr_3$, considering two cases.

\medskip
\noindent\textbf{Case 1.}
The vertex $w$ lies on at least one of the paths $P_{v_1v}$ and $P_{v_2v}$.
Without loss of generality, assume that $w$ lies on $P_{v_1v}$.
Since $v\neq w$, the vertex $w$ is not the terminal vertex of
$P_{v_1v}$.
Let $w_1$ be the out-neighbour of $w$ used by $P_{v_1v}$, i.e., $P_{v_1v}$ traverses the edge $(w,w_1)\in E$.
Let $P_{v_1w}$ denote the subpath of $P_{v_1v}$ from $v_1$ to $w$.
Since $\deg^{+}(w)=2$, there exists another out-neighbour $w_2\neq w_1$ with $(w,w_2)\in E$.
Choose a simple directed path $P_{w_2v}$ from $w_2$ to $v$.
Then we obtain a third first-return path
\[
fr_3:=vv_1\circ P_{v_1w}\circ ww_2\circ P_{w_2v},
\]
of length
\[
|fr_3|=\bigl|\,vv_1\circ P_{v_1w}\circ ww_2\circ P_{w_2v}\,\bigr|=:l_3.
\]
The path $fr_3$ is distinct from $fr_1$, since it uses the edge
$(w,w_2)$ instead of $(w,w_1)$, and it is distinct from $fr_2$,
since their first edges are different.

\medskip
\noindent\textbf{Case 2.}
The vertex $w$ lies on neither $P_{v_1v}$ nor $P_{v_2v}$. Consider the two first-return paths $fr_1$ and $fr_2$ defined above. Since $G$ is strongly connected, $w$ is reachable from every vertex, in particular from every vertex in $V(fr_1)\cup V(fr_2)$. Choose a vertex
\[
u\in V(fr_1)\cup V(fr_2)
\]
such that
\[
\operatorname{dist}(u,w)
=
\min\bigl\{
\operatorname{dist}(x,w)
:
x\in V(fr_1)\cup V(fr_2)
\bigr\}.
\]
Choose $i\in\{1,2\}$ such that $u\in V(fr_i)$, and let $u_1$ be the successor of $u$ on $fr_i$. Let $P_{uw}$ be a shortest directed path from $u$ to $w$. If $P_{uw}$ began with the edge $(u,u_1)$, then
\[
\operatorname{dist}(u,w)
=
1+\operatorname{dist}(u_1,w),
\]
and therefore
\[
\operatorname{dist}(u_1,w)
<
\operatorname{dist}(u,w),
\]
contradicting the choice of $u$, since
$u_1\in V(fr_1)\cup V(fr_2)$. Hence $P_{uw}$ begins with an outgoing edge of $u$ different from $(u,u_1)$.

Moreover, no internal vertex of $P_{uw}$ belongs to
$V(fr_1)\cup V(fr_2)$. Indeed, if an internal vertex
$x$ of $P_{uw}$ belonged to $V(fr_1)\cup V(fr_2)$, then the subpath of $P_{uw}$ from $x$ to $w$ would imply
\[
\operatorname{dist}(x,w)
<
\operatorname{dist}(u,w),
\]
again contradicting the choice of $u$. In particular, $v$ is not an internal vertex of $P_{uw}$.

Let $P_{vu}$ denote the subpath of $fr_i$ from $v$ to $u$, and choose a shortest directed path $P_{wv}$ from $w$ to $v$. We define
\[
fr_3
:=
P_{vu}\circ P_{uw}\circ P_{wv}.
\]
The path $P_{vu}$ does not contain $v$ except at its initial vertex, since it is a subpath of the first-return path $fr_i$. As shown above, $P_{uw}$ does not contain $v$, and, since $P_{wv}$ is a shortest path from $w$ to $v$, it does not contain $v$ before its terminal vertex. Consequently, $fr_3$ is a first-return path to $v$.

Finally, $w\in V(fr_3)$, whereas
\[
w\notin V(fr_1)\cup V(fr_2).
\]
Therefore, $fr_3$ is distinct from both $fr_1$ and $fr_2$. Thus, $v$ admits three distinct first-return paths $fr_1$, $fr_2$, and $fr_3$, as required.
\end{proof}

\begin{lemma}\label{srovnani_m+k_hran_motyl}
Let $\mathcal{SC}_{m+k}(m)$ denote the class of all strongly connected digraphs on $m$ vertices and $m+k$ edges (up to isomorphism), where $k\ge 2$ and $m\geq 2$.
Then
\[
h\!\left(\mathcal{B}_{m-1,m}^{\,m-1}\right)<h(G)
\qquad\text{for every }G\in\mathcal{SC}_{m+k}(m)\text{ with }k\ge 2.
\]
Moreover, if $m\geq 3$, then
\[ 
    h(G)>\frac{\ln2}{m-1} \qquad \text{for every } G\in\mathcal{SC}_{m+k}(m)\text{ with }k\ge 2.
 \]
\end{lemma}

\begin{proof}
Let $m=2$. Then $k\in\{1,2\}$ (no other values of $k$ are possible without multiple edges), and
\[
\mathcal{SC}_{3}(2)=\left\{\mathcal{B}^{\,1}_{1,2}\right\}.
\]
Moreover, $\mathcal{SC}_{4}(2)$ contains only the complete digraph on two vertices; denote it by $K_2$. That is,
\[
V(K_2)=\{v_1,v_2\}, \qquad E(K_2)=V(K_2)\times V(K_2).
\]

By Lemma~\ref{srovnavani_motyl_minimum},
\[
h\!\left(\mathcal{B}^{\,1}_{1,2}\right)=-\ln(z_0)<\ln 2.
\]
On the other hand, Lemma~\ref{lem:R_bounds} yields
\[
h(K_2)=-\ln\!\left(\frac12\right)=\ln 2.
\]

Now, let $m\geq 3$. Fix $k\ge 2$ and let $G\in\mathcal{SC}_{m+k}$. Then one of the following two situations occurs:
\begin{enumerate}
    \item there exists $v\in V$ such that $\deg^+(v)\ge 3$,
    \item there exist distinct vertices $v,w\in V$ such that $\deg^+(v)\ge 2$ and $\deg^+(w)\ge 2$.
\end{enumerate}

\noindent\textbf{Case 1.} Assume that there exists $v\in V$ with $\deg^+(v)\ge 3$.
Then there exist three pairwise distinct vertices $v_1,v_2,v_3\in V$ such that
$(v,v_1),(v,v_2),(v,v_3)\in E$. For each $i\in\{1,2,3\}$, let $P_{v_iv}$ be a shortest directed path from $v_i$ to $v$.
By Lemma~\ref{distance}, at most one of the paths $P_{v_1v},P_{v_2v},P_{v_3v}$ has length $m-1$; hence at least two of them have length strictly smaller than $m-1$.
Without loss of generality, assume that $|P_{v_1v}|<m-1$ and $|P_{v_2v}|<m-1$.
Composing these paths with the edges $(v,v_1)$ and $(v,v_2)$, respectively, we obtain two distinct first-return paths at $v$, say $fr_1$ and $fr_2$, with
\[
|fr_1|=\bigl|\,vv_1\circ P_{v_1v}\bigr|=l_1\le m-1,
\qquad
|fr_2|=\bigl|\,vv_2\circ P_{v_2v}\bigr|=l_2\le m-1.
\]
Moreover, there exists a path $P_{v_3v}$ of length at most $m-1$. Hence there is a third first-return path $fr_3$ satisfying
\[
|fr_3|=\bigl|\,vv_3\circ P_{v_3v}\,\bigr|=l_3\le m.
\]

We now apply Lemma~\ref{compare}. In the present setting,
\[
G(z)=z^{l_1}+z^{l_2}+z^{l_3}
\qquad\text{and}\qquad
Q(z)=\frac{1}{1-G(z)}=\frac{1}{1-z^{l_1}-z^{l_2}-z^{l_3}}.
\]
Let $R_Q$ denote the radius of convergence of $Q(z)$, i.e.\ the unique positive solution of
$1-z^{l_1}-z^{l_2}-z^{l_3}=0$. For every $z\in(0,1)$ (and in particular for $z=R_Q$) we have
\[
1-z^{l_1}-z^{l_2}-z^{l_3}\le 1-2z^{m-1}-z^m< 1-2z^{m-1}.
\]
Consequently,
\[
R_Q< 2^{-\frac{1}{m-1}}.
\]
By Lemma~\ref{compare}, if $R_v$ denotes the radius of convergence of the generating function $P_v(z)$, then
\[
R_v\le R_Q< 2^{-\frac{1}{m-1}},
\]
and therefore
\[
h(G)= -\ln(R_v)> \frac{\ln 2}{m-1}.
\]
Finally, Corollary~\ref{srovnavani_motyl_minimum} yields
$h\!\left(\mathcal{B}_{m-1,m}^{\,m-1}\right)<\frac{\ln 2}{m-1}$, and the claim follows.

\noindent\textbf{Case 2.}
Assume that there exist two vertices $v,w\in V$ such that $\deg^+(v)\ge 2$ and
$\deg^+(w)\ge 2$. Since Case~1 has already been settled, we may restrict ourselves to
the situation where there exist $v,w\in V$ such that $\deg^+(v)=2$ and $\deg^+(w)=2$.
Fix the vertex $v$ and denote by $v_1$ and $v_2$ the vertices satisfying
$(v,v_1),(v,v_2)\in E$. Set
\[
\operatorname{d}(v)=\max_{i=1,2}\{\text{dist}(v_i,v)\}.
\]
If $\operatorname{d}(v)<m-1$, then by Lemma~\ref{lem:three_fr_dv_small} there exist three distinct
first-return paths $fr_1$, $fr_2$, and $fr_3$ at $v$ with lengths
$l_1\leq m-1$, $l_2\leq m-1$, and $l_3\geq 1$, respectively.
Applying Lemma~\ref{compare} exactly as in Case~1, we have
\[
1-G(z)
=
1-z^{l_1}-z^{l_2}-z^{l_3}
\leq
1-2z^{m-1}-z^{l_3}
<
1-2z^{m-1}.
\]
Therefore,
\[
R_v\leq R_Q<2^{-\frac{1}{m-1}},
\]
and hence
\[
h(G)>\frac{\ln 2}{m-1}.
\]

It remains to consider the case $\operatorname{d}(v)=m-1$. Then there exists a vertex
$v_i$, $i=1,2$, such that $\text{dist}(v_i,v)=m-1$. Without loss of generality, assume
that this vertex is $v_1$. Hence, by Lemma~\ref{distance}, there exists a simple path
$P_{v_1v}$ of length $m-1$. Let $w_1$ be the out-neighbour of $w$ such that the path
$P_{v_1v}$ uses the edge $(w,w_1)$. Since $\deg^+(w)=2$, let $w_2$ denote the other
out-neighbour of $w$. Write the path $P_{v_1v}$ as a sequence of vertices
\[
P_{v_1v}=u_1u_2\dots u_{m-1}u_m,
\]
with $u_1=v_1$ and $u_m=v$. Since this path is simple and has length $m-1$, it contains
all remaining vertices of the digraph, in particular $v_2$, $w$, $w_1$, and $w_2$. In
other words, there exist indices $i_1,i_2,i_3\in\{1,2,\dots,m\}$ such that
$u_{i_1}=w$, $u_{i_2}=w_1$, and $u_{i_3}=w_2$, and there exists an index
$i_4\in\{2,3,\dots,m\}$ such that $u_{i_4}=v_2$ (since $v_2\neq v_1$). Since the path
$P_{v_1v}$ uses the edge $(w,w_1)$, we have $i_2=i_1+1$. We distinguish two subcases:
\begin{enumerate}
    \item $i_3>i_2$,
    \item $i_3<i_2$.
\end{enumerate}

\noindent\textbf{Subcase 2.1.}
Assume that $i_3>i_2$. Then, instead of using the edge $(w,w_1)$ in $P_{v_1v}$, we may
use the edge $(w,w_2)$ and thereby shorten the path from $v_1$ to $v$ by at least one
(since we skip at least the vertex $w_1$). This produces a directed path from $v_1$ to
$v$ of length at most $m-2$, contradicting $\text{dist}(v_1,v)=m-1$.

\noindent\textbf{Subcase 2.2.}
Assume that $i_3<i_2$. Let $P_{v_1w}$ denote the subpath of $P_{v_1v}$ that starts at
$v_1$ and ends at $w$, and let $P_{wv}$ denote the subpath that starts at $w$ and ends
at $v$. The fact that $i_3<i_2$ implies that there exists a simple closed path $C_{ww}$ formed
by the edge $(w,w_2)$ together with a subpath of $P_{v_1w}$. Such a path has length
$l_c$, and in particular $l_c\le m-1$, since the vertex $v$ is not contained in it. The
extremal case $l_c=m-1$ may occur, e.g.\ if $w_1=v$ and the cycle visits all vertices
except $v$.

By $(C_{ww})^t$ we mean the concatenation
$C_{ww}\circ C_{ww}\circ\dots\circ C_{ww}$ repeated $t$ times. From these pieces we can
construct infinitely many first-return paths at the vertex $v$ of lengths
$l_t=m+tl_c\le m+t(m-1)$, $t\in\mathds{N}_0$:
\[
|vv_1\circ P_{v_1w}\circ (C_{ww})^t\circ P_{wv}|
=|vv_1|+t|C_{ww}|+|P_{v_1v}|
=1+tl_c+m-1
=m+tl_c\le m+t(m-1).
\]
Moreover, since $P_{v_1w}$ and $P_{wv}$ are subpaths of the simple path $P_{v_1v}$, they
are simple as well, and the vertex $v$ appears only at the end of the subpath $P_{wv}$.
Hence
\[
fr_t:=vv_1\circ P_{v_1w}\circ (C_{ww})^t\circ P_{wv}
\]
is indeed a first-return path for every $t\in\mathds{N}_0$.

We also know that there exists an index $i_4\in\{2,3,\dots,m\}$ such that $u_{i_4}=v_2$.
Therefore, there exists a path $P_{v_2v}$ of length at most $m-2$. Concatenating with the
edge $(v,v_2)$, we obtain a first-return path of length $r_1\le m-1$.

We now define the function $G(z)$:
\[
G(z)=z^{r_1}+\sum_{t=0}^\infty z^{m+tl_c}
=z^{r_1}+\frac{z^m}{1-z^{l_c}}.
\]
Consider the value of $G(z)$ at $z_0=2^{-\frac{1}{m-1}}$. Then $z_0^{m-1}=\frac{1}{2}$
and $z_0^m=\frac{z_0}{2}$. Moreover,
\[
G(z_0)\ge z_0^{m-1}+\frac{z_0^m}{1-z_0^{l_c}}
=\frac{1}{2}+\frac{z_0^m}{1-z_0^{l_c}}.
\]
Since $l_c\le m-1$, we have $1-z^{l_c}\le 1-z^{m-1}$ and thus
$\frac{1}{1-z^{l_c}}\ge \frac{1}{1-z^{m-1}}$. Therefore,
\[
G(z_0)\ge \frac{1}{2}+\frac{z_0^m}{1-z_0^{m-1}}
=\frac{1}{2}+2z_0^m
=\frac{1}{2}+z_0
>\frac{1}{2}+\frac{1}{2}=1.
\]
Consequently,
\[
1-G(z_0)<0,
\]
and hence there exists a solution $r_0\in(0,z_0)$ such that
\[
1-G(r_0)=0.
\]
Since $1-G(z)$ is strictly decreasing on $(0,1)$, $r_0$ is the unique positive singularity of $Q(z)$ in $(0,1)$. Moreover, by definition,
\[
Q(z)=\sum_{n=0}^{\infty} q(n)z^n,
\]
where $q(n)\geq 0$ for all $n$. Hence, by Theorem~\ref{pring}, $r_0$ is the radius of convergence of $Q(z)$. Finally, Lemma~\ref{compare} and Corollary~\ref{srovnavani_motyl_minimum} yield the claim.

\end{proof}

\begin{theorem}\label{thm_global_minimum_motyl}
Let $m\ge 2$ and let $\mathcal{SC}^{h>0}(m)$ denote the class of all strongly connected digraphs
on $m$ vertices (up to isomorphism) with positive topological entropy. Then
\[
h\!\left(\mathcal{B}^{\,m-1}_{m-1,m}\right)\le h(G)
\qquad\text{for every }G\in\mathcal{SC}^{h>0}(m).
\]
Moreover, equality holds if and only if $G\cong \mathcal{B}^{\,m-1}_{m-1,m}$.
\end{theorem}

\begin{proof}
By Lemma~\ref{lem:R_bounds} we have
\[
\mathcal{SC}^{h>0}(m)=\bigcup_{k\ge 1}\mathcal{SC}_{m+k}(m).
\]
By Lemma~\ref{srovnani_m+k_hran_motyl} we have
\[
h\!\left(\mathcal{B}_{m-1,m}^{\,m-1}\right)<h(G)
\qquad\text{for every }G\in\mathcal{SC}_{m+k}(m)\text{ and every }k\ge 2.
\]
On the other hand, Corollary~\ref{srovnavani_motyl_minimum} yields
\[
h\!\left(\mathcal{B}_{m-1,m}^{\,m-1}\right) < h(G)
\quad \text{for every } G\in \mathcal{SC}_{m+1}(m)\smallsetminus \left\{\mathcal{B}_{m-1,m}^{\,m-1}\right\}.
\]
Therefore,
\[
h\!\left(\mathcal{B}_{m-1,m}^{\,m-1}\right)\le h(G)
\qquad\text{for every }G\in\mathcal{SC}^{h>0}(m),
\]
and equality holds if and only if $G\cong \mathcal{B}_{m-1,m}^{\,m-1}$.
\end{proof}

\begin{corollary}
Let $\varepsilon>0$. Define
\[
h_{\min}(m):=\min\{\,h(G): G\in \mathcal{SC}^{h>0}(m)\,\}.
\]
Then the minimal integer $m$ such that $h_{\min}(m)\le \varepsilon$
(equivalently, such that $h\!\left(\mathcal{B}^{\,m-1}_{m-1,m}\right)\le~\varepsilon$) is
\[
m_{\min}(\varepsilon)=\max\left\{2,\left\lceil\frac{\ln(e^\varepsilon+1)}{\varepsilon}\right\rceil\right\}.
\]
\end{corollary}

\begin{proof}
For $m\ge 2$, the minimum in $\mathcal{SC}^{h>0}(m)$ is attained by $\mathcal{B}^{\,m-1}_{m-1,m}$, whose
topological polynomial is $1-z^{m-1}-z^m$. Let $R_m\in(0,1)$ be its unique positive root. Hence $h\left(\mathcal{B}^{\,m-1}_{m-1,m}\right)\le \varepsilon$ is equivalent to
$R_m\ge e^{-\varepsilon}$.

Since $1-z^{m-1}-z^m$ is strictly decreasing on $(0,1)$, the inequality
$R_m\ge e^{-\varepsilon}$ holds if and only if 
\[
1-e^{-\varepsilon(m-1)}-e^{-\varepsilon m}\ge 0.
\]
Rearranging gives
\[
e^{-\varepsilon(m-1)}\le \frac{1}{1+e^{-\varepsilon}}.
\]
Taking logarithms and simplifying yields
\[
m\ge \frac{\ln(1+e^{-\varepsilon})}{\varepsilon}+1
=\frac{\ln(e^\varepsilon+1)}{\varepsilon}.
\]
Therefore the minimal integer $m$ satisfying $h_{\min}(m)\le \varepsilon$ is
\[
m_{\min}(\varepsilon)=\max\left\{2,\left\lceil\frac{\ln(e^\varepsilon+1)}{\varepsilon}\right\rceil\right\}.
\]
\end{proof}

\section{A pyramidal entropy diagram for $\mathcal{SC}_{m+1}(m)$}
In this section we study the class $\mathcal{SC}_{m+1}(m)$ of all strongly connected digraphs
with $m$ vertices and $m+1$ edges.
Lemma~\ref{lem:m+1_is_butterfly} states that every such digraph can be described as a
$(t,k_1,k_2)$-butterfly digraph $\mathcal{B}^{\,t}_{k_1,k_2}$ satisfying
\begin{equation}\label{vztah_k_1_k_2}
    k_1+k_2-t=m.
\end{equation}
For fixed $m$ and every $t\in\{1,2,\dots,m-1\}$ we define the subclasses
\[ 
    \mathcal{B}^{\,t}:=\left\{\mathcal{B}^{\,s}_{k_1,k_2}\in  \mathcal{SC}_{m+1}(m)\,|\,s=t\right\}.
 \]

We may assume without loss of generality that $k_1\le k_2$, since the digraphs
$\mathcal{B}^{\,t}_{k_1,k_2}$ and $\mathcal{B}^{\,t}_{k_2,k_1}$ are isomorphic. Moreover, for fixed $m$ we represent the digraph $\mathcal{B}^{\,t}_{k_1,k_2}$ by the ordered
pair $(k_1,k_2)$. We can do this, since for every such ordered pair there exists a unique, uniquely determined $t$ by \eqref{vztah_k_1_k_2}, and hence a uniquely determined digraph $\mathcal{B}^{\,t}_{k_1,k_2}$ and, for all $t\in\{1,2,\dots,m-1\}$, the classes $\mathcal{B}^{\,t}$ are pairwise disjoint.

\begin{lemma}\label{mnoziny_k_1k_2}
Let $m,t\in \mathds{N}$, $m\ge 2$, and $t\le m-1$. Define
\begin{align*}
    K_1^t&:=\left\{k_1\in\mathds{N}\,|\,\exists k_2\in\mathds{N}:\mathcal{B}^{\,t}_{k_1,k_2}\in\mathcal{B}^{\,t}\right\}, \\
    K_2^t&:=\left\{k_2\in\mathds{N}\,|\,\exists k_1\in\mathds{N}:\mathcal{B}^{\,t}_{k_1,k_2}\in\mathcal{B}^{\,t}\right\}.
\end{align*}
Then
\begin{align*}
    K_1^t&=\left\{t,t+1,\dots,\left\lfloor\frac{m+t}{2}\right\rfloor-1, \left\lfloor\frac{m+t}{2}\right\rfloor\right\}, \\
    K_2^t&=\left\{m,m-1,\dots,\left\lceil\frac{m+t}{2}\right\rceil+1,\left\lceil\frac{m+t}{2}\right\rceil\right\},
\end{align*}
and
\[ 
    \left|\mathcal{B}^{\,t}\right|=\left\lfloor\frac{m-t}{2}\right\rfloor+1.
 \]
Moreover,
\[ 
    \mathcal{SC}_{m+1}(m)=\bigcup_{t=1}^{m-1}\mathcal{B}^{\,t}.
 \]
\end{lemma}

\begin{proof}
The relation $k_1\le k_2$ together with \eqref{vztah_k_1_k_2} yields
\[ 
    2k_1\le m+t,
 \]
and since $k_1$ is a natural number, we obtain
\[ 
    k_1\le\left\lfloor\frac{m+t}{2}\right\rfloor.
 \]
Moreover, the condition $1\le t\le \min\{k_1,k_2\}$ implies $t\le k_1$.
Therefore,
\[ 
    k_1\in K_1^t=\left\{t,t+1,\dots,\left\lfloor\frac{m+t}{2}\right\rfloor-1, \left\lfloor\frac{m+t}{2}\right\rfloor\right\},
 \]
and for each $k_1\in K_1^t$ there exists a unique $k_2$ satisfying \eqref{vztah_k_1_k_2}.
Consequently,
\[ 
    k_2\in K_2^t=\left\{m+t-k_1\,|\,k_1\in K_1^t\right\}
    =\left\{m,m-1,\dots,\left\lceil\frac{m+t}{2}\right\rceil+1,\left\lceil\frac{m+t}{2}\right\rceil\right\},
 \]
and hence
\[ 
    \left|\mathcal{B}^{\,t}\right|
    =\left|K_2^t\right|
    =\left|K_1^t\right|
    =\left\lfloor\frac{m+t}{2}\right\rfloor-t+1
    =\left\lfloor\frac{m-t}{2}\right\rfloor+1.
 \]

To prove the relation
\[
    \mathcal{SC}_{m+1}(m)=\bigcup_{t=1}^{m-1}\mathcal{B}^{\,t},
\]
it suffices to determine an upper bound for the parameter $t$, which is determined in Lemma~\ref{t_omez}.

\end{proof}

\begin{corollary}
Let $m\ge 2$, and let $\mathcal{SC}_{m+1}(m)$ be defined as above. Then
\[
\left|\mathcal{SC}_{m+1}(m)\right|
=\left\lfloor\frac{(m-1)^2}{4}\right\rfloor + (m-1).
\]
\end{corollary}

\begin{proof}
By Lemma~\ref{mnoziny_k_1k_2} we have
\[
\mathcal{SC}_{m+1}(m)=\bigcup_{t=1}^{m-1}\mathcal{B}^{\,t}
\qquad\text{and}\qquad
\left|\mathcal{B}^{\,t}\right|=\left\lfloor\frac{m-t}{2}\right\rfloor+1.
\]
Hence,
\begin{align*}
\left|\mathcal{SC}_{m+1}(m)\right|
&=\sum_{t=1}^{m-1}\left(\left\lfloor\frac{m-t}{2}\right\rfloor+1\right) \\
&=\sum_{t=1}^{m-1}\left\lfloor\frac{t}{2}\right\rfloor + (m-1).
\end{align*}

\medskip
\noindent\textbf{Case 1.} Suppose $m-1=2n$ for some $n\in\mathbb{N}$. Pair the terms $(2k-1,2k)$ for $k=1,\dots,n$ to obtain
\[
\sum_{t=1}^{2n}\left\lfloor\frac{t}{2}\right\rfloor
=\sum_{k=1}^{n}\left(\left\lfloor\frac{2k-1}{2}\right\rfloor+\left\lfloor\frac{2k}{2}\right\rfloor\right)
=\sum_{k=1}^{n}(2k-1)
=n^2.
\]
Therefore,
\[
\sum_{t=1}^{m-1}\left\lfloor\frac{t}{2}\right\rfloor
=\frac{(m-1)^2}{4}=\left\lfloor\frac{(m-1)^2}{4}\right\rfloor,
\]
since $m-1$ is even.

\medskip
\noindent\textbf{Case 2.} Suppose $m-1=2n-1$ for some $n\in\mathbb{N}$. Then
\[
\sum_{t=1}^{2n-1}\left\lfloor\frac{t}{2}\right\rfloor
=\sum_{t=1}^{2n}\left\lfloor\frac{t}{2}\right\rfloor-\left\lfloor\frac{2n}{2}\right\rfloor
=n^2-n.
\]
On the other hand,
\[
\left\lfloor\frac{(m-1)^2}{4}\right\rfloor
=\left\lfloor\frac{(2n-1)^2}{4}\right\rfloor
=\left\lfloor n^2-n+\frac14\right\rfloor
=n^2-n,
\]
so in this case as well
\[
\sum_{t=1}^{m-1}\left\lfloor\frac{t}{2}\right\rfloor
=\left\lfloor\frac{(m-1)^2}{4}\right\rfloor.
\]

Combining the two cases yields
\[
\left|\mathcal{SC}_{m+1}(m)\right|
=\left\lfloor\frac{(m-1)^2}{4}\right\rfloor + (m-1),
\]
as claimed.
\end{proof}

\begin{lemma}\label{polynomy_koren}
For every $n_1,n_2\in\mathds{N}$ with $n_1\leq n_2$ and for every $z\in(0,1)$ it holds, that
\[ 
    1-z^{n_1}-z^{n_2}>1-z^{n_1-1}-z^{n_2+1}.
 \]
\end{lemma}

\begin{proof}
Let $n_1,n_2\in\mathds{N}$ with $n_1\leq n_2$ and let $z\in(0,1)$. Then
\begin{align*}
    z^{n_1-1}&>z^{n_2}, \\
    z^{n_1-1}(1-z)&>z^{n_2}(1-z), \\
    z^{n_1-1} - z^{n_1}&>z^{n_2}-z^{n_2+1}, \\
    - z^{n_1}- z^{n_2}&>-z^{n_1-1}-z^{n_2+1}, \\
    1-z^{n_1}- z^{n_2}&>1-z^{n_1-1}-z^{n_2+1}.
\end{align*}
\end{proof}

\begin{definition}
    Let $\mathcal{SC}^{h>0}(m)$ be defined as above and $G_1,G_2\in\mathcal{SC}^{h>0}(m)$. We define a binary relation $<_h$ on $\mathcal{SC}^{h>0}(m)$ by
\[
G_1 <_h G_2 \quad:\Longleftrightarrow\quad h(G_1) < h(G_2).
\]
\end{definition}

\begin{lemma}\label{horizont}
Let $m\in \mathds{N}$, $m\ge 2$. Then for every $t\in\{1,2,\dots,m-1\}$ we have
\[ 
    \left(\left\lfloor\frac{m+t}{2}\right\rfloor, \left\lceil\frac{m+t}{2}\right\rceil\right) <_h
    \left(\left\lfloor\frac{m+t}{2}\right\rfloor-1, \left\lceil\frac{m+t}{2}\right\rceil+1\right)
    <_h\dots <_h(t+1,m-1)<_h(t,m).
 \]
In other words, for each $t$ the class $\mathcal{B}^{\,t}$ is totally ordered by the relation $<_h$.
\end{lemma}

\begin{proof}
Let $t\in\{1,2,\dots,m-1\}$. For any $\mathcal{B}^{\,t}_{k_1,k_2}\in \mathcal{B}^{\,t}$ there exists a topological polynomial
\[ 
    1-z^{k_1}-z^{k_2},
 \]
whose positive root $z_0(k_1,k_2)$ determines the topological entropy (that is,
$h\!\left(\mathcal{B}^{\,t}_{k_1,k_2}\right)=-\ln(z_0)$). Moreover, by Lemma~\ref{mnoziny_k_1k_2},
\begin{align*}
    k_1&\in\left\{t,t+1,\dots,\left\lfloor\frac{m+t}{2}\right\rfloor-1, \left\lfloor\frac{m+t}{2}\right\rfloor\right\}, \\
    k_2&\in\left\{m,m-1,\dots,\left\lceil\frac{m+t}{2}\right\rceil+1,\left\lceil\frac{m+t}{2}\right\rceil\right\}.
\end{align*}

All digraphs in this class are
\[ 
    \left(\left\lfloor\frac{m+t}{2}\right\rfloor, \left\lceil\frac{m+t}{2}\right\rceil\right),
    \left(\left\lfloor\frac{m+t}{2}\right\rfloor-1, \left\lceil\frac{m+t}{2}\right\rceil+1\right),
    \dots,(t+1,m-1),(t,m),
 \]
and their topological polynomials are
\begin{align*}
    T_1&=1-z^{\left\lfloor\frac{m+t}{2}\right\rfloor}-z^{\left\lceil\frac{m+t}{2}\right\rceil}, \\
    T_2&=1-z^{\left\lfloor\frac{m+t}{2}\right\rfloor-1}-z^{\left\lceil\frac{m+t}{2}\right\rceil+1}, \\
    &\vdots \\
    T_{\left\lfloor\frac{m-t}{2}\right\rfloor}&=1-z^{t+1}-z^{m-1}, \\
    T_{\left\lfloor\frac{m-t}{2}\right\rfloor+1}&=1-z^t-z^m.       
\end{align*}
Let $z_i\in(0,1)$ denote the positive root of $T_i$ for
$i\in \left\{1,2,\dots,\left\lfloor\frac{m-t}{2}\right\rfloor+1\right\}$. Then, by
Lemma~\ref{polynomy_koren}, we have $z_j>z_{j+1}$ for every
$j\in \left\{1,2,\dots,\left\lfloor\frac{m-t}{2}\right\rfloor\right\}$. Consequently, the class
$\mathcal{B}^{\,t}$ is totally ordered by the relation $<_h$, and
\[ 
    \left(\left\lfloor\frac{m+t}{2}\right\rfloor, \left\lceil\frac{m+t}{2}\right\rceil\right) <_h
    \left(\left\lfloor\frac{m+t}{2}\right\rfloor-1, \left\lceil\frac{m+t}{2}\right\rceil+1\right)
    <_h\dots <_h(t+1,m-1)<_h(t,m).
 \]
\end{proof}

\begin{lemma}\label{lem:poly_positive_at_z0}
Let $m \in \mathbb{N}$.
\begin{enumerate}
    \item If $m \geq 5$, then $(m-3, m) <_h (m-2, m-2)$.
    \item If $m \geq 6$, then $(m-4, m) <_h (m-3, m-2)$.
    \item If $m \geq 7$, then $(m-4, m-1) <_h (m-3, m-3)$.
    \item If $m = 7$, then $(m-3, m-3) <_h (m-5, m)$, whereas for all $m \geq 8$, $(m-5, m) <_h (m-3, m-3)$.
    \item If $m = 7$, then $(2, m) <_h (3, m-2)$, whereas for all $m \geq 8$, $(3, m-2) <_h (2, m)$. 
\end{enumerate}
\end{lemma}

\begin{proof}
The proof of these relations relies on the evaluation of the topological polynomial of one family at the unique positive root of the topological equation of another. The signs of the resulting expressions are determined through algebraic factorization and monotonicity analysis. Due to the technical nature of these calculations, the detailed analytical derivation is provided in Appendix.
\end{proof}

\begin{theorem}[Pyramidal entropy diagram for $\mathcal{SC}_{m+1}(m)$]\label{thm:pyramidal_diagram}
Let $m\ge 2$. For each $t\in\{1,2,\dots,m-1\}$ let $\mathcal{B}^{\,t}\subset \mathcal{SC}_{m+1}(m)$
be defined as above, and represent each $\mathcal{B}^{\,t}_{k_1,k_2}\in\mathcal{B}^{\,t}$ by the
ordered pair $(k_1,k_2)$ with $k_1\le k_2$.

Define the \emph{pyramidal diagram} $\mathcal{H}_m$ as the directed graph whose vertex set is
\[
V(\mathcal{H}_m)=\bigcup_{t=1}^{m-1}\mathcal{B}^{\,t}
\quad(\text{identified with their pairs }(k_1,k_2)),
\]
and whose edges are of the following three types:

\begin{enumerate}
\item[\textup{(H)}] (\emph{Horizontal edges within a fixed $t$.})
For each fixed $t$, connect consecutive elements in the chain from Lemma~\ref{horizont}:
\[
\Bigl(\Bigl\lfloor\frac{m+t}{2}\Bigr\rfloor,\Bigl\lceil\frac{m+t}{2}\Bigr\rceil\Bigr)
\to
\Bigl(\Bigl\lfloor\frac{m+t}{2}\Bigr\rfloor-1,\Bigl\lceil\frac{m+t}{2}\Bigr\rceil+1\Bigr)
\to\cdots\to (t+1,m-1)\to (t,m).
\]

\item[\textup{(V)}] (\emph{Vertical/diagonal edges between $t$ and $t-1$.})
Let $t\in\{2,3,\dots,m-1\}$ and let $(k_1,k_2)\in\mathcal{B}^{\,t}$.
Whenever the resulting pairs are admissible (and written in nondecreasing order),
add edges
\[
(k_1,k_2)\to (k_1,k_2-1)
\qquad\text{and}\qquad
(k_1,k_2)\to (k_1-1,k_2).
\]
\item[\textup{($\text{V}^*$)}]
Edges
\[ 
    (m-3,m)\to (m-2,m-2), \, (m-4,m)\to (m-3,m-2), \,(m-4,m-1)\to(m-3,m-3).
 \]
\end{enumerate}

Then every edge of $\mathcal{H}_m$ is strict with respect to entropy:
if $G_1\to G_2$ is an edge in $\mathcal{H}_m$, then
\[
h(G_1)<h(G_2).
\]
Consequently, the reachability relation $\preceq$ induced by $\mathcal{H}_m$,
defined by $G_1 \preceq G_2$ if and only if there exists a directed path from $G_1$ to $G_2$ in $\mathcal{H}_m$,
is a partial order on $\mathcal{SC}_{m+1}(m)$ and it is compatible with $<_{h}$.
Moreover, the Hasse diagram of $(\mathcal{SC}_{m+1}(m),\preceq)$ is obtained as the transitive reduction of $\mathcal{H}_m$. In particular, $\mathcal{H}_m$ has the pyramidal shape shown in Figures~\ref{fig:pyramid-common}, \ref{fig:pyramid-odd-ending} and \ref{fig:pyramid-even-ending}(odd/even $m$).
\end{theorem}

\begin{proof}
For edges of type \textup{(H)}, the entropy-defining polynomials are of the form
$1-z^{k_1}-z^{k_2}$, and the strict inequality along each horizontal step
follows from Lemma~\ref{polynomy_koren}, hence
$h$ strictly increases along the horizontal arrows.
For edges of type \textup{(V)}, we decrease exactly one exponent by $1$, so the corresponding
polynomials satisfy the monotonicity assumption of Lemma~\ref{monotonicity}. Therefore the unique positive
root in $(0,1)$ strictly decreases, and thus the entropy strictly increases. Edges $(\text{V}^*)$ follow from Lemma~\ref{lem:poly_positive_at_z0}.
This proves $h(G_1)<h(G_2)$ for every arrow $G_1\to G_2$ in $\mathcal{H}_m$. 

Since every edge $G_1\to G_2$ satisfies $h(G_1)<h(G_2)$, the digraph $\mathcal{H}_m$ is acyclic.
Therefore the reachability relation $\preceq$ is antisymmetric and hence defines a
partial order on $V(\mathcal{H}_m)=\mathcal{SC}_{m+1}(m)$, and it is compatible with $<_{h}$.
The claim about the Hasse diagram follows by definition: it is the transitive reduction
of the reachability digraph.

\end{proof}

\newcommand{\TikzExternalOff}{\ifcsname tikzexternaldisable\endcsname\tikzexternaldisable\fi}
\newcommand{\TikzExternalOn}{\ifcsname tikzexternalenable\endcsname\tikzexternalenable\fi}

\makeatletter
\newcommand{\SetTLabel}[2]{\expandafter\gdef\csname TLabel@#1\endcsname{#2}}
\newcommand{\TLabel}[1]{%
  \ifcsname TLabel@#1\endcsname
    \csname TLabel@#1\endcsname
  \else
    \text{(doplním)}%
  \fi
}

\newcommand{\SetNodeLabel}[2]{\expandafter\gdef\csname NodeLabel@#1\endcsname{#2}}
\newcommand{\NodeLabel}[1]{%
  \ifcsname NodeLabel@#1\endcsname
    \csname NodeLabel@#1\endcsname
  \else
    #1%
  \fi
}
\makeatother

\newcommand{\PyramidalEntropyTopOnly}[2]{%
\begin{tikzpicture}[
  x=1cm, y=1cm,
  every node/.style={font=\small, inner sep=0pt, outer sep=0pt},
  edge/.style={-{Stealth[length=3.2mm,width=2.2mm]}, line width=0.55pt}
]
  \def\xsep{2.75}
  \def\ysep{1.05}
  \def\labelsep{1.2}

  \pgfmathtruncatemacro{\m}{#1}
  \pgfmathtruncatemacro{\rows}{\m-1}
  \pgfmathtruncatemacro{\toprows}{min(8,\rows)}

  \foreach \r in {1,...,\toprows} {%
    \pgfmathtruncatemacro{\cols}{int(ceil((\r+1)/2))}%
    \pgfmathsetmacro{\y}{-(\r-1)*\ysep}%

    \node[anchor=east] at (-\labelsep,\y) {$t=\TLabel{\r}:$};

    \foreach \c in {1,...,\cols} {%
      \pgfmathsetmacro{\x}{(\c-1)*\xsep}%
      \node (N-\r-\c) at (\x,\y) {\NodeLabel{N-\r-\c}};
    }%

    \pgfmathtruncatemacro{\colsMinusOne}{\cols-1}%
    \ifnum\colsMinusOne>0
      \foreach \c in {1,...,\colsMinusOne} {%
        \pgfmathtruncatemacro{\cp}{\c+1}%
        \draw[edge] (N-\r-\c) -- (N-\r-\cp);
      }%
    \fi
  }%

  \pgfmathtruncatemacro{\toprowsMinusOne}{\toprows-1}%
  \ifnum\toprowsMinusOne>0
    \foreach \r in {1,...,\toprowsMinusOne} {%
      \pgfmathtruncatemacro{\rp}{\r+1}%
      \pgfmathtruncatemacro{\colsA}{int(ceil((\r+1)/2))}%
      \pgfmathtruncatemacro{\colsB}{int(ceil((\rp+1)/2))}%
      \pgfmathtruncatemacro{\colsMin}{min(\colsA,\colsB)}%
      \foreach \c in {1,...,\colsMin} {%
        \draw[edge] (N-\r-\c) -- (N-\rp-\c);
      }%
    }%
  \fi

\pgfmathtruncatemacro{\colsLast}{int(ceil((\toprows+1)/2))}%
\pgfmathsetmacro{\yvdots}{-(\toprows-0.2)*\ysep}
\foreach \c in {1,...,\colsLast} {%
  \pgfmathsetmacro{\x}{(\c-1)*\xsep}%
  \node at (\x,\yvdots) {$\vdots$};
}%

  #2
\end{tikzpicture}%
}

\newcommand{\PyramidalEntropyBottomOnly}[2]{%
\begin{tikzpicture}[
  x=1cm, y=1cm,
  every node/.style={font=\small, inner sep=0pt, outer sep=0pt},
  edge/.style={-{Stealth[length=3.2mm,width=2.2mm]}, line width=0.55pt}
]
  \def\xsep{2.75}
  \def\ysep{1.05}
  \def\labelsep{1.2}

  \pgfmathsetmacro{\xstart}{1*\xsep}

  \pgfmathsetmacro{\yvdots}{-0.50*\ysep}
  \pgfmathsetmacro{\yPen}{-1.5*\ysep}
  \pgfmathsetmacro{\yLast}{-2.5*\ysep}

  \foreach \c in {1,...,3} {%
    \pgfmathsetmacro{\x}{\xstart+(\c-1)*\xsep}%
    \node at (\x,\yvdots) {$\vdots$};
  }%

  \node[anchor=east] at (-\labelsep,\yPen)  {$t=\TLabel{P}:$};
  \node[anchor=east] at (-\labelsep,\yLast) {$t=\TLabel{L}:$};

  \node at (\xstart-\xsep,\yPen)  {$\dots$};
  \node at (\xstart-\xsep,\yLast) {$\dots$};

  \foreach \c in {1,...,4} {%
    \pgfmathsetmacro{\x}{\xstart+(\c-1)*\xsep}%
    \node (P-\c) at (\x,\yPen)  {\NodeLabel{P-\c}};
    \node (L-\c) at (\x,\yLast) {\NodeLabel{L-\c}};
  }%

  \foreach \c in {1,2,3} {%
    \pgfmathtruncatemacro{\cp}{\c+1}%
    \draw[edge] (P-\c) -- (P-\cp);
    \draw[edge] (L-\c) -- (L-\cp);
  }%

  \foreach \c in {1,...,4} {%
    \draw[edge] (P-\c) -- (L-\c);
  }%

  #2
\end{tikzpicture}%
}

\begin{figure}[H]
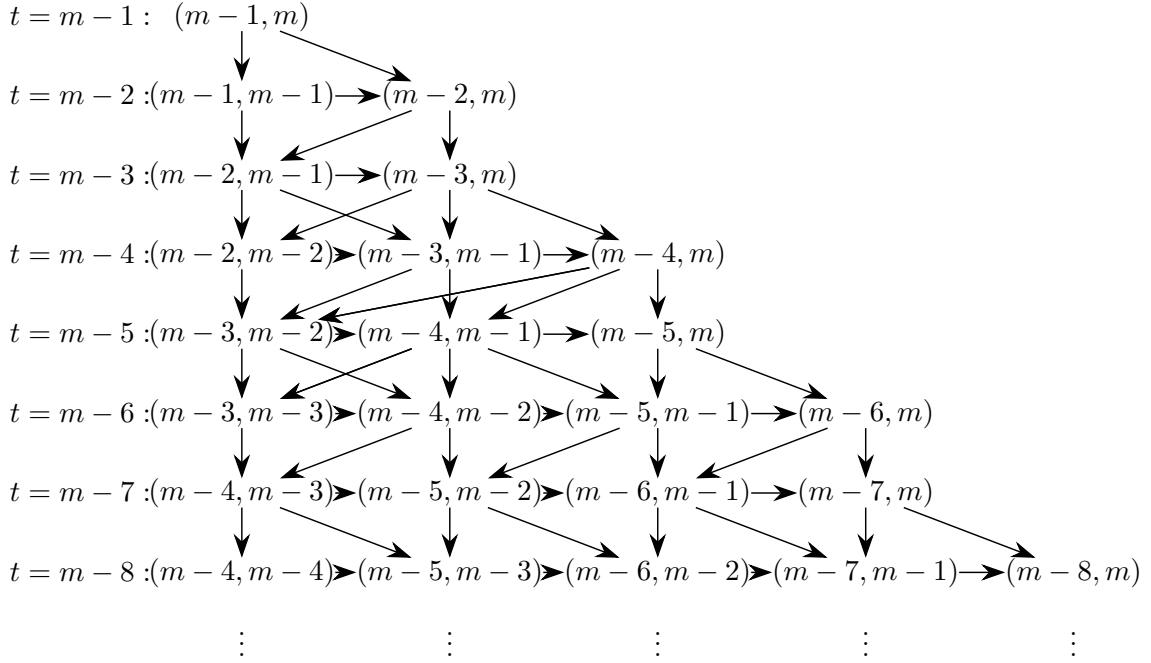

\centering
\TikzExternalOff
{
  \SetTLabel{1}{m-1}
  \SetTLabel{2}{m-2}
  \SetTLabel{3}{m-3}
  \SetTLabel{4}{m-4}
  \SetTLabel{5}{m-5}
  \SetTLabel{6}{m-6}
  \SetTLabel{7}{m-7}
  \SetTLabel{8}{m-8}

  \SetNodeLabel{N-1-1}{$(m-1,m)$}
  \SetNodeLabel{N-2-1}{$(m-1,m-1)$}
  \SetNodeLabel{N-2-2}{$(m-2,m)$}
  \SetNodeLabel{N-3-1}{$(m-2,m-1)$}
  \SetNodeLabel{N-3-2}{$(m-3,m)$}
  \SetNodeLabel{N-4-1}{$(m-2,m-2)$}
  \SetNodeLabel{N-4-2}{$(m-3,m-1)$}
  \SetNodeLabel{N-4-3}{$(m-4,m)$}
  \SetNodeLabel{N-5-1}{$(m-3,m-2)$}
  \SetNodeLabel{N-5-2}{$(m-4,m-1)$}
  \SetNodeLabel{N-5-3}{$(m-5,m)$}
  \SetNodeLabel{N-6-1}{$(m-3,m-3)$}
  \SetNodeLabel{N-6-2}{$(m-4,m-2)$}
  \SetNodeLabel{N-6-3}{$(m-5,m-1)$}
  \SetNodeLabel{N-6-4}{$(m-6,m)$}
  \SetNodeLabel{N-7-1}{$(m-4,m-3)$}
  \SetNodeLabel{N-7-2}{$(m-5,m-2)$}
  \SetNodeLabel{N-7-3}{$(m-6,m-1)$}
  \SetNodeLabel{N-7-4}{$(m-7,m)$}
  \SetNodeLabel{N-8-1}{$(m-4,m-4)$}
  \SetNodeLabel{N-8-2}{$(m-5,m-3)$}
  \SetNodeLabel{N-8-3}{$(m-6,m-2)$}
  \SetNodeLabel{N-8-4}{$(m-7,m-1)$}
  \SetNodeLabel{N-8-5}{$(m-8,m)$}

  \PyramidalEntropyTopOnly{11}{%
    \draw[edge] (N-1-1) -- (N-2-2);
    \draw[edge] (N-2-2) -- (N-3-1);
    \draw[edge] (N-3-1) -- (N-4-2);
    \draw[edge] (N-4-2) -- (N-5-1);
    \draw[edge] (N-5-1) -- (N-6-2);

    \draw[edge] (N-3-2) -- (N-4-3);
    \draw[edge] (N-4-3) -- (N-5-2);
    \draw[edge] (N-5-2) -- (N-6-3);
    \draw[edge] (N-5-3) -- (N-6-4);

    \draw[edge] (N-3-2) -- (N-4-1);
    \draw[edge] (N-4-3) -- (N-5-1);
    \draw[edge] (N-5-2) -- (N-6-1);
    \draw[edge] (N-4-3) -- (N-5-1); 
    \draw[edge] (N-5-2) -- (N-6-1); 
    \draw[edge] (N-6-2) -- (N-7-1); 
    \draw[edge] (N-6-3) -- (N-7-2); 
    \draw[edge] (N-6-4) -- (N-7-3); 
    \draw[edge] (N-7-1) -- (N-8-2); 
    \draw[edge] (N-7-2) -- (N-8-3); 
    \draw[edge] (N-7-3) -- (N-8-4); 
    \draw[edge] (N-7-4) -- (N-8-5);
  }%
}
\TikzExternalOn
\caption{Pyramidal entropy diagram for $\mathcal{SC}_{m+1}(m)$ (common part).}
\label{fig:pyramid-common}
\end{figure}

\begin{figure}[H]
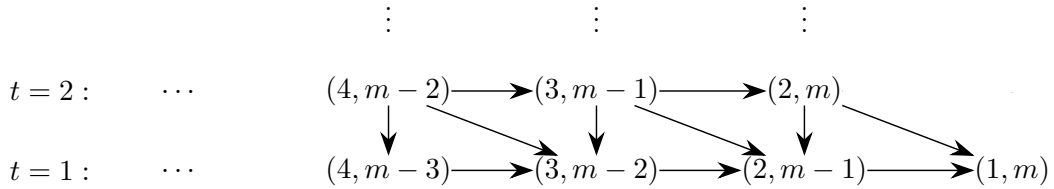

\centering
\TikzExternalOff
{
  \SetTLabel{P}{2}
  \SetTLabel{L}{1}

  \SetNodeLabel{L-4}{$(1,m)$}
  \SetNodeLabel{L-3}{$(2,m-1)$}
  \SetNodeLabel{L-2}{$(3,m-2)$}
  \SetNodeLabel{L-1}{$(4,m-3)$}

  \SetNodeLabel{P-3}{$(2,m)$}
  \SetNodeLabel{P-2}{$(3,m-1)$}
  \SetNodeLabel{P-1}{$(4,m-2)$}
  \SetNodeLabel{P-4}{}

  \PyramidalEntropyBottomOnly{11}{%
    \draw[white, line width=8pt] (P-3) -- (P-4);
    \draw[white, line width=8pt] (P-4) -- (L-4);

    \draw[edge] (P-1) -- (L-2);
    \draw[edge] (P-2) -- (L-3);
    \draw[edge] (P-3) -- (L-4);
  }%
}
\TikzExternalOn
\caption{Pyramidal entropy diagram for $\mathcal{SC}_{m+1}(m)$ (odd $m$, ending).}
\label{fig:pyramid-odd-ending}
\end{figure}

\begin{figure}[H]
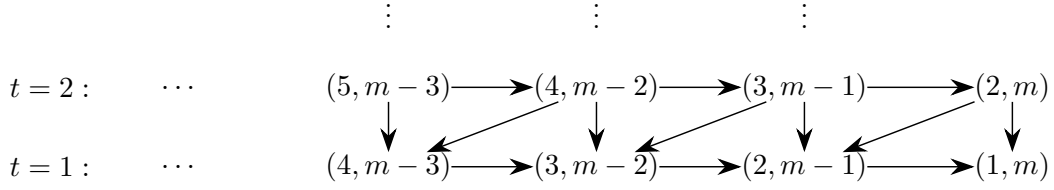

\centering
\TikzExternalOff
{
  \SetTLabel{P}{2}
  \SetTLabel{L}{1}

  \SetNodeLabel{L-4}{$(1,m)$}
  \SetNodeLabel{L-3}{$(2,m-1)$}
  \SetNodeLabel{L-2}{$(3,m-2)$}
  \SetNodeLabel{L-1}{$(4,m-3)$}

  \SetNodeLabel{P-4}{$(2,m)$}
  \SetNodeLabel{P-3}{$(3,m-1)$}
  \SetNodeLabel{P-2}{$(4,m-2)$}
  \SetNodeLabel{P-1}{$(5,m-3)$}

  \PyramidalEntropyBottomOnly{12}{%
    \draw[edge] (P-2) -- (L-1);
    \draw[edge] (P-3) -- (L-2);
    \draw[edge] (P-4) -- (L-3);
  }%
}
\TikzExternalOn
\caption{Pyramidal entropy diagram for $\mathcal{SC}_{m+1}(m)$ (even $m$, ending).}
\label{fig:pyramid-even-ending}
\end{figure}

\begin{corollary}\label{7_min}
Let $m\ge 2$. In the class $\mathcal{SC}_{m+1}(m)$, the first ten entropy minima are attained by
\begin{align*}
    &(m-1,m),\ (m-1,m-1),\ (m-2,m),\ (m-2,m-1),\ (m-3,m)\\
    &(m-2,m-2),\ (m-3,m-1),\ (m-4,m),\ (m-3,m-2),\ (m-4,m-1),
\end{align*}
in this order. Moreover, the two entropy maxima in $\mathcal{SC}_{m+1}(m)$ are attained by
\[
(1,m)\quad\text{and}\quad (2,m-1),
\]
where $(1,m)$ is the unique maximizer and $(2,m-1)$ is the second maximizer.
\end{corollary}

\begin{proof}
By the construction of the pyramidal entropy diagrams (Figures~\ref{fig:pyramid-common}--\ref{fig:pyramid-even-ending}),
an oriented path in the diagram corresponds to a strict inequality with respect to $<_h$.
Hence, from each of the ten minima listed above there is a directed path to every other vertex of the diagram that is not among these minima,
which implies that each of them is $<_h$-smaller than all remaining digraphs in $\mathcal{SC}_{m+1}(m)$. 

Conversely, each of the two listed maxima is reachable from every other vertex, and therefore it is $<_h$-greater
than all remaining digraphs in $\mathcal{SC}_{m+1}(m)$.

The mutual relations among the minima and maxima are also apparent from the diagram.
\end{proof}

\begin{corollary}
    The segment consisting of the initial ten entropy minima and the two entropy maxima identified in Corollary~\ref{7_min} is maximal in the sense that it represents the largest set of digraphs in $\mathcal{SC}_{m+1}(m)$ whose relative $<_h$-ordering remains invariant for all $m \ge 2$.
\end{corollary}

\begin{proof}
    By Lemma~\ref{lem:poly_positive_at_z0}, the eleventh entropy minimum is attained by the digraph $(m-3,m-3)$ when $m=7$, whereas for $m \ge 8$, the eleventh minimum is given by $(m-5,m)$. 
    According to the same lemma, the third entropy maximum is the digraph $(3,m-2)$ for $m=7$, while for $m \ge 8$, the third maximum is $(2,m)$. 
    This establishes the maximality of both the initial segment of entropy minima and the segment of entropy maxima.
\end{proof}

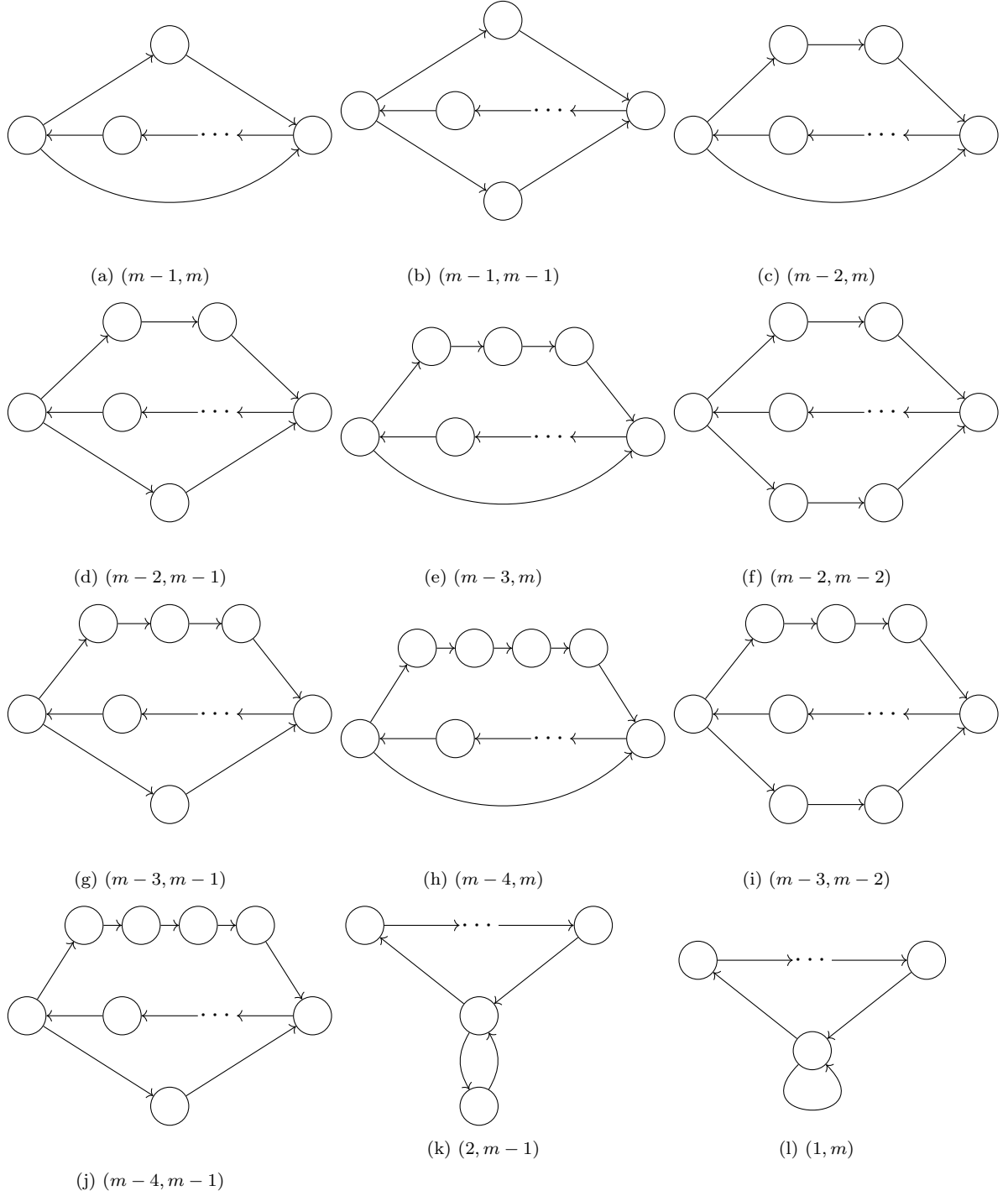
\begin{figure}[H]
\centering

\begin{subfigure}[t]{0.3\textwidth}
\centering
\input{min1.tikz}
\caption{$(m-1,m)$}
\label{fig:min1}
\end{subfigure}\hfill
\begin{subfigure}[t]{0.3\textwidth}
\centering
\input{min2.tikz}
\caption{$(m-1,m-1)$}
\label{fig:min2}
\end{subfigure}\hfill
\begin{subfigure}[t]{0.3\textwidth}
\centering
\input{min3.tikz}
\caption{$(m-2,m)$}
\label{fig:min3}
\end{subfigure}

\vspace{2mm}

\begin{subfigure}[t]{0.3\textwidth}
\centering
\input{min4.tikz}
\caption{$(m-2,m-1)$}
\label{fig:min4}
\end{subfigure}\hfill
\begin{subfigure}[t]{0.3\textwidth}
\centering
\input{min5.tikz}
\caption{$(m-3,m)$}
\label{fig:min5}
\end{subfigure}\hfill
\begin{subfigure}[t]{0.3\textwidth}
\centering
\input{min6.tikz}
\caption{$(m-2,m-2)$}
\label{fig:min6}
\end{subfigure}

\vspace{2mm}

\begin{subfigure}[t]{0.3\textwidth}
\centering
\input{min7.tikz}
\caption{$(m-3,m-1)$}
\label{fig:min7}
\end{subfigure}\hfill
\begin{subfigure}[t]{0.3\textwidth}
\centering
\input{min8.tikz}
\caption{$(m-4,m)$}
\label{fig:min2}
\end{subfigure}\hfill
\begin{subfigure}[t]{0.3\textwidth}
\centering
\input{min9.tikz}
\caption{$(m-3,m-2)$}
\label{fig:min2}
\end{subfigure}\hfill

\vspace{2mm}

\begin{subfigure}[t]{0.3\textwidth}
\centering
\input{min10.tikz}
\caption{$(m-4,m-1)$}
\label{fig:max2}
\end{subfigure}\hfill
\begin{subfigure}[t]{0.3\textwidth}
\centering
\input{max2.tikz}
\caption{$(2,m-1)$}
\label{fig:min2}
\end{subfigure}\hfill
\begin{subfigure}[t]{0.3\textwidth}
\centering
\input{max1.tikz}
\caption{$(1,m)$}
\label{fig:max1}
\end{subfigure}

\caption{Ten entropy minima ((a)--(j)) and two entropy maxima ((k),(l)) of the class $\mathcal{SC}_{m+1}(m)$, for fixed $m\geq 2$.}
\label{fig:entropy-extremals-3x3}
\end{figure}

\begin{corollary}
Let $m\in \{2,3,4,5,6,7\}$. Then the transitive reduction of $\mathcal{H}_m$ describes a total ordering of the class
$\mathcal{SC}_{m+1}(m)$ with respect to the relation $<_h$.
\end{corollary}

\begin{proof}
For $m=2$ there is only one digraph.
For $m\in\{3,4,5\}$ all digraphs correspond to a single minimum from Corollary~\ref{7_min}, and hence they are totally ordered.
For $m=6$ there are exactly $11$ pairwise non-isomorphic digraphs: ten of them correspond to the ten minima in Corollary~\ref{7_min}, and the remaining one corresponds to the maximum from the same corollary.
For $m=7$, there exist exactly $15$ pairwise non-isomorphic digraphs. Ten of these correspond to the minima and two correspond to the maxima identified in Corollary~\ref{7_min}. It remains to compare the three digraphs $(2,7)$, $(4,4)$, and $(3,5)$ and determine their relationship to the tenth minimum $(3,6)$ and the second maximum $(2,6)$. Lemma \ref{lem:poly_positive_at_z0} yields
\[ 
(4,4) <_h (2,7) <_h (3,5),
 \]
while it also holds that $(3,6) <_h (3,5)$ and $(2,7) <_h (2,6)$. Therefore,
\[ 
(3,6) <_h (4,4) <_h (2,7) <_h (3,5) <_h (2,6),
 \]
and the ordering is total
\end{proof}

\begin{corollary}
Let $m\ge 3$ and let $\mathcal{SC}^{h>0}(m)$ be defined as above. Then the butterfly digraph
$\mathcal{B}^{\,m-2}_{m-1,m-1}$ satisfies
\[
h\!\left(\mathcal{B}^{\,m-2}_{m-1,m-1}\right)<h(G)
\quad\text{for every}\quad
G\in \mathcal{SC}^{h>0}(m)\setminus\left\{\mathcal{B}^{\,m-1}_{m-1,m},\,\mathcal{B}^{\,m-2}_{m-1,m-1}\right\}.
\]
In other words, $\mathcal{B}^{\,m-2}_{m-1,m-1}$ attains the second smallest positive topological entropy in the class
$\mathcal{SC}^{h>0}(m)$.
\end{corollary}

\begin{proof}
By Corollary~\ref{7_min}, the type $(m-1,m-1)$ attains the second minimum within the class
$\mathcal{SC}_{m+1}(m)$.
Moreover, in Corollary~\ref{cor:butterfly_t_irrelevant} we showed that
\[
h\!\left(\mathcal{B}^{\,m-2}_{m-1,m-1}\right)=\frac{\ln 2}{m-1},
\]
and by Lemma~\ref{srovnani_m+k_hran_motyl}, for every
$G\in \mathcal{SC}_{m+k}(m)$ with $k\ge 2$,
\[
h(G)>\frac{\ln 2}{m-1}.
\]
Combining these facts yields the claim.
\end{proof}

\section{Conclusion}
Using a generating-function approach, we recover the known first and
second positive-entropy minimizers from the previously studied
loopless spectral setting and provide new analytic--combinatorial
proofs of their extremality and uniqueness. Our formulation also
allows loops and covers the relevant small orders not included in
the previous results.

The principal new contributions concern the internal entropy
structure of the class $\mathcal{SC}_{m+1}(m)$. We introduce a unified
$(t,k_1,k_2)$-butterfly parametrization and show that the topological
polynomial, and hence the entropy, can be read directly from the pair
$(k_1,k_2)$. For fixed $m$, the entropy is a complete invariant up to
isomorphism within this class. We also derive an explicit formula for
the minimum order required to realize a positive entropy not exceeding
a prescribed threshold.

A key outcome of our structural analysis is the reachability order
encoded by the Pyramidal Entropy Diagram. We show that the entropy
ordering is total for $m\leq 7$ and undergoes its first structural
bifurcation at $m=8$, where the relative order of certain butterfly
types reverses. We further determine maximal stable initial and
terminal segments, consisting of ten entropy minima and two entropy
maxima, respectively. Their relative ordering is independent of $m$,
whereas no larger initial or terminal segment has this stability.

\section{Acknowledgements}
Research was funded by institutional support for the development of research organisations (I\v{C} 47813059) and by Grant SGS 16/2024.

\bibliographystyle{plain}
\bibliography{refs}
\pagebreak

\section{Appendix}
\subsection{Proof of Lemma~\ref{lem:poly_positive_at_z0}}
\begin{enumerate}
\item \begin{proof}
Let $z_0=2^{-1/(m-2)}$ be the root of the polynomial $1-2z^{m-2}$, that is, $1=2z_0^{m-2}$.
We evaluate the polynomial $1-z^{m-3}-z^{m}$ at $z_0$:
\begin{align*}
    1-z_0^{m-3}-z_0^{m}
    &=2z_0^{m-2}-z_0^{m-3}-z_0^{m}\\
    &=z_0^{m-3}\bigl(2z_0-1-z_0^3\bigr)\\
    &=z_0^{m-3}(1-z_0)(z_0^2+z_0-1).
\end{align*}
Since $z_0\in(0,1)$, the sign is determined solely by the factor $z_0^2+z_0-1$.
The polynomial $z^2+z-1$ has a unique positive root $\frac{\sqrt5-1}{2}$, so it suffices to verify that
$z_0>\frac{\sqrt5-1}{2}$.
Set $z_0(m)=2^{-1/(m-2)}$. For $m\ge 5$ we have
\[
z_0(5)=2^{-1/3}>\frac{\sqrt5-1}{2}.
\]
Moreover, $z_0(m)$ is increasing in $m$, hence $z_0(m)\ge z_0(5)$ for all $m\ge 5$.
Therefore $z_0(m)>\frac{\sqrt5-1}{2}$ for every $m\ge 5$.
\end{proof}

\item \begin{proof}
Let $z_0\in(0,1)$ satisfy $1-z_0^{m-3}-z_0^{m-2}=0$. Then $1=z_0^{m-3}+z_0^{m-2}$.
We evaluate the polynomial $1-z^{m-4}-z^{m-1}$ at $z_0$:
\begin{align*}
    1-z_0^{m-4}-z_0^{m-1}
    &=z_0^{m-3}+z_0^{m-2}-z_0^{m-4}-z_0^{m-1}\\
    &=z_0^{m-4}\bigl(z_0+z_0^2-1-z_0^4\bigr)\\
    &=z_0^{m-4}(1-z_0)(z_0^3+z_0^2-1).
\end{align*}
Since $z_0^{m-4}>0$ and $1-z_0>0$ on $(0,1)$, the sign is determined solely by the factor
$z_0^3+z_0^2-1$.
Because $m\ge 6$, we have $m-3\ge 3$ and $m-2\ge 4$, hence
\[
1=z_0^{m-3}+z_0^{m-2}\le z_0^3+z_0^4<z_0^3+z_0^2,
\]
which implies $z_0^3+z_0^2-1>0$.
\end{proof}

\item \begin{proof}
Let $z_0=2^{-1/(m-3)}$ be the root of the polynomial $1-2z^{m-3}$, that is, $1=2z_0^{m-3}$.
Again, we evaluate the polynomial $1-z^{m-4}-z^{m-1}$ at $z_0$:
\begin{align*}
    1-z_0^{m-4}-z_0^{m-1}
    &=2z_0^{m-3}-z_0^{m-4}-z_0^{m-1}\\
    &=z_0^{m-4}\bigl(2z_0-1-z_0^3\bigr)\\
    &=z_0^{m-4}(1-z_0)(z_0^2+z_0-1).
\end{align*}
Since $z_0\in(0,1)$, the sign is determined solely by the factor $z_0^2+z_0-1$.
The polynomial $z^2+z-1$ has a unique positive root $\frac{\sqrt5-1}{2}$, so it suffices to verify that
$z_0>\frac{\sqrt5-1}{2}$.
Set $z_0(m)=2^{-1/(m-3)}$. For $m\ge 7$ we have
\[
z_0(7)=2^{-1/4}>\frac{\sqrt5-1}{2}.
\]
Moreover, $z_0(m)$ is increasing in $m$, hence $z_0(m)\ge z_0(7)$ for all $m\ge 7$.
Therefore $z_0(m)>\frac{\sqrt5-1}{2}$ for every $m\ge 7$.
\end{proof}

\item \begin{proof}
Let $z_0=2^{-1/(m-3)}$ be the root of the polynomial $1-2z^{m-3}$, that is, $1=2z_0^{m-3}$.
Again, we evaluate the polynomial $1-z^{m-5}-z^{m}$ at $z_0$:
\begin{align*}
    1-z_0^{m-5}-z_0^{m}
    &=2z_0^{m-3}-z_0^{m-5}-z_0^{m}\\
    &=z_0^{m-5}\bigl(2z_0^2-1-z_0^5\bigr)\\
    &=z_0^{m-5}(1-z_0)\bigl(z_0^4+z_0^3+z_0^2-z_0-1\bigr).
\end{align*}
Since $z_0\in(0,1)$, the sign is determined solely by the factor
\[
f(z_0):=z_0^4+z_0^3+z_0^2-z_0-1.
\]
Set $z_0(m)=2^{-1/(m-3)}$. Direct evaluations yield $f(z_0(7))<0$ and $f(z_0(8))>0$.
Furthermore, since $z_0(8)=2^{-1/5}>\tfrac12$, since $z_0(m)$ is increasing in $m$, and since $f$ is increasing on
$[1/2,1)$, we obtain $f(z_0(m))>0$ for all $m\ge 8$.
\end{proof}
\item \begin{proof}
    Let $z_0 \in (0, 1)$ be the unique positive root of the polynomial $1 - z^3 - z^{m-2}$, which implies the identity $z_0^{m-2} = 1 - z_0^3$. To determine the sign of the polynomial $1 - z^2 - z^m$ at this root, we evaluate:
    \begin{align*}
        1 - z_0^2 - z_0^m &= 1 - z_0^2 - z_0^2(1 - z_0^3) \\ 
        &= z_0^5 - 2z_0^2 + 1 \\
        &=(1 - z_0)(1 + z_0 - z_0^2 - z_0^3 - z_0^4).
    \end{align*}
    Since $z_0 \in (0, 1)$, the sign of the expression is uniquely determined by the factor$$g(z_0) := 1 + z_0 - z_0^2 - z_0^3 - z_0^4.$$ For the case $m=7$, we utilize the identity $z_0^5 = 1 - z_0^3$ to obtain\[ g(z_0) = (1 - z_0^3) + z_0 - z_0^2 - z_0^4 = z_0^5 + z_0 - z_0^2 - z_0^4 = z_0(z_0^3 - 1)(z_0 - 1) > 0. \]
    
    In contrast, let $z_0(m)$ denote the root of the polynomial $1 - z^3 - z^{m-2} = 0$. For $m=8$, identifying $z_0(8)$ requires solving $1 - z^3 - z^6 = 0$ as a quadratic in $z^3$, which yields the exact value $$z_0(8) = \sqrt[3]{\frac{\sqrt{5}-1}{2}}.$$ Direct substitution shows that $g(z_0(8)) < 0$. Since $z_0(m)$ is a strictly increasing function of $m$ and $g(z)$ is strictly decreasing on the interval $[1/2, 1)$, it follows that $g(z_0(m)) < 0$ for all $m \ge 8$.
\end{proof}
\end{enumerate}

\end{document}

%% file: min1.tikz
\begin{tikzpicture}[scale=1.5,
  every node/.style={
    circle,
    draw,
    minimum size=6mm, 
    inner sep=0.2pt
  }
]
  \node (A) at (0,0) {};
  \node (B) at (1,0) {};
  \node (D) at (3,0) {};
  \node[
    shape=circle,
    minimum size=6mm,
    draw=none,
    inner sep=0pt
  ] (dotsM) at ($(B)!0.5!(D)$) {$\cdots$};
  \node (T) at (1.5,0.95) {};
  \draw[->] (A) -- (T);
  \draw[->] (T) -- (D);
  \draw[->] (A) to[bend right=45] (D);
  \draw[->] (D) -- (dotsM);
  \draw[->] (dotsM) -- (B);
  \draw[->] (B) -- (A);
\end{tikzpicture}

%% file: min2.tikz
\begin{tikzpicture}[scale=1.5,
  every node/.style={
    circle,
    draw,
    minimum size=6mm,
    inner sep=0.2pt
  }
]
  \node (A) at (0,0) {};
  \node (B) at (1,0) {};
  \node (D) at (3,0) {};
  \node[
    shape=circle,
    minimum size=6mm,
    draw=none,
    inner sep=0pt
  ] (dotsM) at ($(B)!0.5!(D)$) {$\cdots$};
  \node (T) at (1.5,0.95) {};
  \node (S) at (1.5,-0.95) {};
  \draw[->] (A) -- (T);
  \draw[->] (T) -- (D);
  \draw[->] (A) -- (S);
  \draw[->] (S) -- (D);
  \draw[->] (D) -- (dotsM);
  \draw[->] (dotsM) -- (B);
  \draw[->] (B) -- (A);

\end{tikzpicture}

%% file: min3.tikz
\begin{tikzpicture}[scale=1.5,
  every node/.style={
    circle,
    draw,
    minimum size=6mm,
    inner sep=0.2pt
  }
]
  \node (A) at (0,0) {};
  \node (B) at (1,0) {};
  \node (D) at (3,0) {};
  \node[
    shape=circle,
    minimum size=6mm,
    draw=none,
    inner sep=0pt
  ] (dotsM) at ($(B)!0.5!(D)$) {$\cdots$};
  \node (T1) at (1,0.95) {};
  \node (T2) at (2,0.95) {};
  \draw[->] (A) -- (T1);
  \draw[->] (T1) -- (T2);
  \draw[->] (T2) -- (D);
  \draw[->] (A) to[bend right=45] (D);
  \draw[->] (D) -- (dotsM);
  \draw[->] (dotsM) -- (B);
  \draw[->] (B) -- (A);
\end{tikzpicture}

%% file: min4.tikz
\begin{tikzpicture}[scale=1.5,
  every node/.style={circle,draw,minimum size=6mm,inner sep=0.2pt}
]
  \node (A) at (0,0) {};
  \node (B) at (1,0) {};
  \node (D) at (3,0) {};
  \node[draw=none, inner sep=0pt, minimum size=6mm] (dotsM) at ($(B)!0.5!(D)$) {$\cdots$};
  \def\overhang{0} 
  \def\yTop{0.95}
  \def\yBot{-0.95}
  \def\nTop{2}
  \def\nBot{1}
  \pgfmathsetmacro{\L}{3 + 2*\overhang}
  \pgfmathsetmacro{\xLeft}{0 - \overhang}
  \foreach \i in {1,...,\nTop} {
    \pgfmathsetmacro{\xi}{\xLeft + \i*\L/(\nTop+1)}
    \node (T\i) at (\xi,\yTop) {};
  }
  \foreach \i in {1,...,\nBot} {
    \pgfmathsetmacro{\xi}{\xLeft + \i*\L/(\nBot+1)}
    \node (S\i) at (\xi,\yBot) {};
  }
  \draw[->] (A) -- (T1);
  \foreach \i [evaluate=\i as \j using int(\i+1)] in {1,...,\numexpr\nTop-1\relax} {
    \draw[->] (T\i) -- (T\j);
  }
  \draw[->] (T\nTop) -- (D);
  \draw[->] (A) -- (S1);
  \draw[->] (S\nBot) -- (D);
  \draw[->] (D) -- (dotsM);
  \draw[->] (dotsM) -- (B);
  \draw[->] (B) -- (A);
\end{tikzpicture}

%% file: min5.tikz
\begin{tikzpicture}[scale=1.5,
  every node/.style={
    circle, draw,
    minimum size=6mm,
    inner sep=0.2pt
  }
]
  \node (A) at (0,0) {};
  \node (B) at (1,0) {};
  \node (D) at (3,0) {};
  \node[draw=none, inner sep=0pt, minimum size=6mm] (dotsM)
    at ($(B)!0.5!(D)$) {$\cdots$};
  \def\nTop{3}       
  \def\yTop{0.95}    
  \def\overhang{0}  

  \pgfmathsetmacro{\L}{3 + 2*\overhang}         
  \pgfmathsetmacro{\xLeft}{0 - \overhang}      
  \foreach \i in {1,...,\nTop} {
    \pgfmathsetmacro{\xi}{\xLeft + \i*\L/(\nTop+1)}
    \node (T\i) at (\xi,\yTop) {};
  }
  \draw[->] (A) -- (T1);
  \foreach \i [evaluate=\i as \j using int(\i+1)] in {1,...,\numexpr\nTop-1\relax} {
    \draw[->] (T\i) -- (T\j);
  }
  \draw[->] (T\nTop) -- (D);
  \draw[->] (A) to[bend right=45] (D);
  \draw[->] (D) -- (dotsM);
  \draw[->] (dotsM) -- (B);
  \draw[->] (B) -- (A);
\end{tikzpicture}

%% file: min6.tikz
\begin{tikzpicture}[scale=1.5,
  every node/.style={circle,draw,minimum size=6mm,inner sep=0.2pt}
]
  \node (A) at (0,0) {};
  \node (B) at (1,0) {};
  \node (D) at (3,0) {};
  \node[draw=none, inner sep=0pt, minimum size=6mm] (dotsM) at ($(B)!0.5!(D)$) {$\cdots$};
  \def\overhang{0}
  \def\yTop{0.95}
  \def\yBot{-0.95}
  \def\nTop{2}
  \def\nBot{2}
  \pgfmathsetmacro{\L}{3 + 2*\overhang}
  \pgfmathsetmacro{\xLeft}{0 - \overhang}
  \foreach \i in {1,...,\nTop} {
    \pgfmathsetmacro{\xi}{\xLeft + \i*\L/(\nTop+1)}
    \node (T\i) at (\xi,\yTop) {};
  }
  \foreach \i in {1,...,\nBot} {
    \pgfmathsetmacro{\xi}{\xLeft + \i*\L/(\nBot+1)}
    \node (S\i) at (\xi,\yBot) {};
  }
  \draw[->] (A) -- (T1);
  \foreach \i [evaluate=\i as \j using int(\i+1)] in {1,...,\numexpr\nTop-1\relax} {
    \draw[->] (T\i) -- (T\j);
  }
  \draw[->] (T\nTop) -- (D);
  \draw[->] (A) -- (S1);
  \foreach \i [evaluate=\i as \j using int(\i+1)] in {1,...,\numexpr\nBot-1\relax} {
    \draw[->] (S\i) -- (S\j);
  }
  \draw[->] (S\nBot) -- (D);
  \draw[->] (D) -- (dotsM);
  \draw[->] (dotsM) -- (B);
  \draw[->] (B) -- (A);
\end{tikzpicture}

%% file: min7.tikz
\begin{tikzpicture}[scale=1.5,
  every node/.style={circle,draw,minimum size=6mm,inner sep=0.2pt}
]
  \node (A) at (0,0) {};
  \node (B) at (1,0) {};
  \node (D) at (3,0) {};
  \node[draw=none, inner sep=0pt, minimum size=6mm] (dotsM) at ($(B)!0.5!(D)$) {$\cdots$};
  \def\overhang{0}
  \def\yTop{0.95}
  \def\yBot{-0.95}
  \def\nTop{3}
  \def\nBot{1}
  \pgfmathsetmacro{\L}{3 + 2*\overhang}
  \pgfmathsetmacro{\xLeft}{0 - \overhang}
  \foreach \i in {1,...,\nTop} {
    \pgfmathsetmacro{\xi}{\xLeft + \i*\L/(\nTop+1)}
    \node (T\i) at (\xi,\yTop) {};
  }
  \foreach \i in {1,...,\nBot} {
    \pgfmathsetmacro{\xi}{\xLeft + \i*\L/(\nBot+1)}
    \node (S\i) at (\xi,\yBot) {};
  }
  \draw[->] (A) -- (T1);
  \foreach \i [evaluate=\i as \j using int(\i+1)] in {1,...,\numexpr\nTop-1\relax} {
    \draw[->] (T\i) -- (T\j);
  }
  \draw[->] (T\nTop) -- (D);
  \draw[->] (A) -- (S1);
  \draw[->] (S\nBot) -- (D);
  \draw[->] (D) -- (dotsM);
  \draw[->] (dotsM) -- (B);
  \draw[->] (B) -- (A);
\end{tikzpicture}

%% file: min8.tikz
\begin{tikzpicture}[scale=1.5,
  every node/.style={circle,draw,minimum size=6mm,inner sep=0.2pt}
]
  \node (A) at (0,0) {};
  \node (B) at (1,0) {};
  \node (D) at (3,0) {};
  \node[draw=none, inner sep=0pt, minimum size=6mm] (dotsM) at ($(B)!0.5!(D)$) {$\cdots$};
  \def\overhang{0}  
  \def\yTop{0.95}
  \def\nTop{4}
  \pgfmathsetmacro{\L}{3 + 2*\overhang}
  \pgfmathsetmacro{\xLeft}{0 - \overhang}
  \foreach \i in {1,...,\nTop} {
    \pgfmathsetmacro{\xi}{\xLeft + \i*\L/(\nTop+1)}
    \node (T\i) at (\xi,\yTop) {};
  }
  \draw[->] (A) -- (T1);
  \foreach \i [evaluate=\i as \j using int(\i+1)] in {1,...,\numexpr\nTop-1\relax} {
    \draw[->] (T\i) -- (T\j);
  }
  \draw[->] (T\nTop) -- (D);
  \draw[->, bend right=45] (A) to (D);
  \draw[->] (D) -- (dotsM);
  \draw[->] (dotsM) -- (B);
  \draw[->] (B) -- (A);
\end{tikzpicture}

%% file: min9.tikz
\begin{tikzpicture}[scale=1.5,
  every node/.style={circle,draw,minimum size=6mm,inner sep=0.2pt}
]
  \node (A) at (0,0) {};
  \node (B) at (1,0) {};
  \node (D) at (3,0) {};
  \node[draw=none, inner sep=0pt, minimum size=6mm] (dotsM) at ($(B)!0.5!(D)$) {$\cdots$};
  \def\overhang{0} 
  \def\yTop{0.95}
  \def\yBot{-0.95}
  \def\nTop{3}
  \def\nBot{2}
  \pgfmathsetmacro{\L}{3 + 2*\overhang}
  \pgfmathsetmacro{\xLeft}{0 - \overhang}
  \foreach \i in {1,...,\nTop} {
    \pgfmathsetmacro{\xi}{\xLeft + \i*\L/(\nTop+1)}
    \node (T\i) at (\xi,\yTop) {};
  }
  \foreach \i in {1,...,\nBot} {
    \pgfmathsetmacro{\xi}{\xLeft + \i*\L/(\nBot+1)}
    \node (S\i) at (\xi,\yBot) {};
  }
  \draw[->] (A) -- (T1);
  \foreach \i [evaluate=\i as \j using int(\i+1)] in {1,...,\numexpr\nTop-1\relax} {
    \draw[->] (T\i) -- (T\j);
  }
  \draw[->] (T\nTop) -- (D);
  \draw[->] (A) -- (S1);
  \foreach \i [evaluate=\i as \j using int(\i+1)] in {1,...,\numexpr\nBot-1\relax} {
    \draw[->] (S\i) -- (S\j);
  }
  \draw[->] (S\nBot) -- (D);
  \draw[->] (D) -- (dotsM);
  \draw[->] (dotsM) -- (B);
  \draw[->] (B) -- (A);
\end{tikzpicture}

%% file: min10.tikz
\begin{tikzpicture}[scale=1.5,
  every node/.style={circle,draw,minimum size=6mm,inner sep=0.2pt}
]
  \node (A) at (0,0) {};
  \node (B) at (1,0) {};
  \node (D) at (3,0) {};
  \node[draw=none, inner sep=0pt, minimum size=6mm] (dotsM) at ($(B)!0.5!(D)$) {$\cdots$};
  \def\overhang{0}  
  \def\yTop{0.95}
  \def\yBot{-0.95}
  \def\nTop{4}
  \def\nBot{1}
  \pgfmathsetmacro{\L}{3 + 2*\overhang}
  \pgfmathsetmacro{\xLeft}{0 - \overhang}
  \foreach \i in {1,...,\nTop} {
    \pgfmathsetmacro{\xi}{\xLeft + \i*\L/(\nTop+1)}
    \node (T\i) at (\xi,\yTop) {};
  }
  \foreach \i in {1,...,\nBot} {
    \pgfmathsetmacro{\xi}{\xLeft + \i*\L/(\nBot+1)}
    \node (S\i) at (\xi,\yBot) {};
  }
  \draw[->] (A) -- (T1);
  \foreach \i [evaluate=\i as \j using int(\i+1)] in {1,...,\numexpr\nTop-1\relax} {
    \draw[->] (T\i) -- (T\j);
  }
  \draw[->] (T\nTop) -- (D);
  \draw[->] (A) -- (S1);
  \draw[->] (S\nBot) -- (D);
  \draw[->] (D) -- (dotsM);
  \draw[->] (dotsM) -- (B);
  \draw[->] (B) -- (A);
\end{tikzpicture}

%% file: max2.tikz
\begin{tikzpicture}[scale=1.5,
  every node/.style={circle,draw,minimum size=6mm,inner sep=0.2pt}
]
  \def\a{1.2}   
  \def\yTop{0.95}
  \def\yBot{-0.95}
  \node (C) at (0,0) {};
  \node (T1) at (-\a,\yTop) {};
  \node (T2) at ( \a,\yTop) {};
  \node[draw=none, inner sep=0pt, minimum size=6mm] (dotsT) at (0,\yTop) {$\cdots$};
  \draw[->] (C)   -- (T1);
  \draw[->] (T1)  -- (dotsT);
  \draw[->] (dotsT) -- (T2);
  \draw[->] (T2)  -- (C);
  \node (S) at (0,\yBot) {};
  \draw[->] (C) to[out=-120,in=120,looseness=1.1] (S);
  \draw[->] (S) to[out=60,in=-60,looseness=1.1] (C);
\end{tikzpicture}

%% file: max1.tikz
\begin{tikzpicture}[scale=1.5,
  every node/.style={circle,draw,minimum size=6mm,inner sep=0.2pt}
]
  \def\a{1.2}
  \def\yTop{0.95}
  \node (C) at (0,0) {};
  \node (T1) at (-\a,\yTop) {};
  \node (T2) at ( \a,\yTop) {};
  \node[draw=none, inner sep=0pt, minimum size=6mm] (dotsT) at (0,\yTop) {$\cdots$};
  \draw[->] (C) -- (T1);
  \draw[->] (T1) -- (dotsT);
  \draw[->] (dotsT) -- (T2);
  \draw[->] (T2) -- (C);
  \draw[->] (C) to[out=-135,in=-45,looseness=8] (C);
\end{tikzpicture}